\documentclass{compositio}

\usepackage{amsmath, mathtools}

\input{kmacros3.sty}
\theoremstyle{plain}
\newtheorem{maintheorem}{Main Theorem}

\DeclareMathOperator{\Ass}{Ass}

\DeclareMathOperator{\sep}{sep}

\DeclareMathOperator{\Bl}{Bl}

\DeclareMathOperator{\RR}{\mathcal{R}}
\DeclareMathOperator{\CC}{\mathbb{C}}
\DeclareMathOperator{\MM}{\mathcal{M}}

\DeclareMathOperator{\Max}{Max}

\newcommand{\fm}{\mathfrak{m}}
\newcommand{\fa}{\mathfrak{a}}

\newcommand{\NN}{\mathbb{N}}

\usepackage{mabliautoref}

\begin{document}

\title{Zariski-Nagata Theorems for Singularities and the Uniform Izumi-Rees Property}
\author{Thomas Polstra}
\email{tmpolstra@ua.edu}
\address{Department of Mathematics, University of Alabama, Tuscaloosa, AL 35487 USA}
\thanks{The author was supported in part by NSF Grant DMS \#2101890 and NSF Grand DMS \#2502317 during the development of this article.}

\begin{abstract}
We introduce and explore the Uniform Izumi-Rees Property in Noetherian rings with applications to multiplicity theory and containment relationships among symbolic powers of ideals. As an application, we prove that if $R$ is a normal domain essentially of finite type over a field, there exists a constant $C$ so that for all prime ideals $\fp\subseteq \fq\in\Spec(R)$, if $\fp\subseteq \fq^{(t)}$, then for all $n\in\NN$, there is a containment of symbolic powers $\fp^{(Cn)}\subseteq \fq^{(tn)}$.
\end{abstract}

\maketitle

\section{Introduction}

Let $R$ be a commutative excellent domain and $\Spec(R)$ the collection of prime ideals of $R$. If $\fp\in\Spec(R)$ then the \emph{$n$th symbolic power of $\fp$} is the ideal $\fp^{(n)}:=\fp^nR_\fp\cap R$. If $0\not=f\in R$ then the \emph{order of $f$ at $\fp$} is $\ord_\fp(f)=\sup\{n\in \NN\mid f\in\fp^{(n)}\}$. Let $e(R_\fp/fR_\fp)$ denote the (Hilbert-Samuel) multiplicity of the local ring $R_\fp/fR_\fp$ with respect to the maximal ideal. The values $\ord_\fp(f)$ and $e(R_\fp/fR_\fp)$ serve as competing notions of vanishing order of the function $f$ along the generic point of $V(\fp)$, coinciding if $R_\fp$ is non-singular. An important distinction between multiplicity and order, at singular primes, is that multiplicity enjoys a semi-continuity property, if $\fp\subseteq \fq\in\Spec(R)$, then $e(R_\fp/fR_\fp)\leq e(R_\fq/fR_\fq)$. Consequently, semi-continuity of multiplicity provides the Local Zariski-Nagata Theorem for primes $\fp\subseteq \fq\in\Spec(R)$ belonging to the non-singular locus of $\Spec(R)$; if $R_\fq$ is non-singular then for every $n\in\NN$ there is a containment of ideals $\fp^{(n)}\subseteq \fq^{(n)}$.

If $\fp\subseteq \fq\in\Spec(R)$ and $R_\fq$ is singular, then it is possible that there exists $f\in\fp$ so that $\ord_\fp(f)>\ord_\fq(f)$, equivalently there exists an $n\in\NN$ so that $\fp^{(n)}\not\subseteq \fq^{(n)}$. The Uniform Chevalley Theorem, \cite[Theorem~2.3]{HKV}, is an adaptation of the Local Zariski-Nagata Theorem to a singular point $\fq\in\Spec(R)$, and provides a constant $C$, depending on $\fq$, so that for all $\fp\subseteq \fq$ and $n\in\NN$ there is a containment of ideals $\fp^{(Cn)}\subseteq \fq^{(n)}$.

\begin{theorem}
\label{theorem Zariski-nagata}
    Let $R$ be a Noetherian ring of arbitrary characteristic.
    \begin{enumerate}
        \item\label{local Zariski nagata} \cite[Local Zariski-Nagata Theorem, Page~143]{NagataLocalRings}: If $\fp\subseteq \fq$ are prime ideals belonging to the non-singular locus of $R$ then $\fp^{(n)}\subseteq \fq^{(n)}$ for all $n\in\mathbb{N}$.
        \item\label{Uniform Chevalley HKV} \cite[Uniform Chevalley Theorem]{HKV}: Let $\fq\in\Spec(R)$ a prime with the property that $R_\fq$ is analytically unramified. There exists a constant $C$, depending on $\fq$, so that if $\fp\in\Spec(R_\fq)$,  then $\fp^{(Cn)}\subseteq \fq^{(n)}$ for all $n\in\NN$.
        \item\label{Uniform chevalley direct summands}\cite[Main Result]{DSGJ} Let $R$ be a graded direct summand of either the polynomial ring over a field or a discrete valuation ring. Suppose that the degree of the generators of $R$ as an algebra are bounded by $D$, and $\MM$ the homogeneous maximal ideal of $R$. If $\fp\in\Spec(R)$ then $\fp^{(Dn)}\subseteq \MM^n$ for all $n\in\NN$.
    \end{enumerate}
\end{theorem}

Our first main theorem is in the spirit of Theorem~\ref{theorem Zariski-nagata} (\ref{Uniform chevalley direct summands}) and greatly generalizes (\ref{Uniform Chevalley HKV}). If $R$ is a normal domain essentially of finite type over an algebraically closed field, $\fp\subseteq \fq \in \Spec(R)$, then we can give specific information on a constant $C$ so that $\fp^{(Cn)}\subseteq \fq^{(n)}$ for all $n\in\NN$.

\begin{maintheorem}
    \label{Main Theorem eft over an algebraically closed field}
    Let $k$ be an algebraically closed field and $R$ a normal domain essentially of finite type over $k$. Let $X\subseteq \mathbb{P}^n_k$ be an arithmetically normal projective closure of $\Spec(R)$, $S$ the coordinate ring of $X$, and $e(S)$ the Hilbert-Samuel multiplicity of $S$ with respect to its homogeneous maximal ideal. Then for all $\fp\subseteq \fq\in\Spec(R)$, for all $n\in\NN$, there is a containment of ideals
        \[
        \fp^{(e(S)n+1)}\subseteq \fq^{(n)}.
        \]
\end{maintheorem}

Our investigations are not limited to algebras over an algebraically closed field and some of our methods take inspiration from the proof of (\ref{Uniform Chevalley HKV}) presented in \cite[Theorem~2.3]{HKV}. We rely on several theorems of Rees from \cite{ReesValuationsAssociatedToIdeals, ReesDegree, ReesIzumisTheorem} in our studies. A corollary of Rees' theorems is that if $R$ is locally analytically irreducible and $\fq\in\Spec(R)$, then there exists a constant $C$, which may depend on $\fq$, such that for any $0\not=f\in \fq$, $e(R_\fq/fR_\fq)\leq C\ord_\fq(f)$. The reference used in this context that makes it unclear if the constant $C$ can be chosen independently of $\fq$ is the Izumi-Rees Theorem from \cite[Theorem~C]{ReesIzumisTheorem}. Essential to our investigations is the following ``Uniform Izumi-Rees Theorem.''

\begin{maintheorem}[Uniform Izumi-Rees Theorem]
    \label{Main Theorem Uniform Izumi Rees}
    Let $k$ be a field and $R$ a normal domain essentially of finite type over $k$. Then $R$ enjoys \textbf{the Uniform Izumi-Rees Property}, there exists a constant $C$ so that for all $\fq\in\Spec(R)$ if $0\not =f\in\fq $, then
    \[
    e(R_\fq/fR_\fq)\leq C\ord_\fq(f).
    \]
\end{maintheorem}

\begin{remark}
    See Corollary~\ref{corollary multiple definitions of Uniform Izumi-Rees Property} and Definition~\ref{definition uniform izumi rees} for an equivalent characterization of the Uniform Izumi-Rees Property.
\end{remark}

Significant progress has been made in uniformly comparing the powers, symbolic powers, and integral powers of ideals in regular rings. If $R$ is a regular ring, then $R$ enjoys the Uniform Symbolic Topology Property. If $I\subseteq R$ an ideal, and $h$ the maximal height of an associated prime of $I$. Then $I^{(hn)}\subseteq I^n$ for all $n\in\NN$, \cite{ELS,HHComparison,MaSchwedeSymbolic,MuryamaSymbolic}. Solutions to the Uniform Symbolic Topology Property Problem in regular rings have been transformative with wide-reaching connections between multiplier/test ideal theory, closure operations, perfectoid spaces, and big Cohen-Macaulay algebras in rings of all characteristics, c.f. \cite{HHAnnals, TakagiYoshida, MaSchwedeSymbolic, Dietz, RG, BhattCM}. The Uniform Symbolic Property for regular rings implies the following improvement of the Local Zariski-Nagata Theorem.

\begin{theorem}[{Corollary of the Uniform Symbolic Topology Theorem, \cite{ELS,HHComparison,MaSchwedeSymbolic,MuryamaSymbolic}}]
\label{theorem ustp with application to Zariski-Nagata}
Let $R$ be an excellent Noetherian domain of finite Krull dimension $d$. Then for all $\fp\subseteq \fq\in\Spec(R)$, if $R_\fq$ is non-singular and $\fp\subseteq \fq^{(t)}$, then $\fp^{(dn)}\subseteq \fq^{(tn)}$ for all $n\in\NN$. 
\end{theorem}

Our next main result is analogous to Theorem~\ref{theorem ustp with application to Zariski-Nagata}, applicable to any normal domain essentially of finite type over a field, and is a generalization of Theorem~\ref{theorem Zariski-nagata} (\ref{Uniform Chevalley HKV}).

\begin{maintheorem}[Improved Uniform Chevalley Theorem]
    \label{Main Theorem Improved Uniform Chevalley general}
    Let $k$ be a field and $R$ a normal domain essentially of finite type over $k$. There exists a constant $C$ so that for all primes $\fp \subseteq \fq$, if $\fp\subseteq \fq^{(t)}$, then for every $n\in\NN$ there is a containment of ideals
        \[
        \fp^{(Cn)}\subseteq \fq^{(tn)}.
        \]
\end{maintheorem}

The paper is organized as follows. Section~\ref{Section valuations} contains preliminary materials on Rees valuations, multiplicity, symbolic powers, and the Izumi-Rees Theorem. Main Theorem~\ref{Main Theorem eft over an algebraically closed field} and Main Theorem~\ref{Main Theorem Uniform Izumi Rees} are proven in Section~\ref{Section Uniform Izumi}. Main Theorem~\ref{Main Theorem Improved Uniform Chevalley general} is then derived as an application of Main Theorem~\ref{Main Theorem Uniform Izumi Rees} and equicharacteristic multiplier/test ideal theory in Section~\ref{Section Uniform Chevalley}. 

\section*{Acknowledgements}
The author thanks Hanlin Cai, Dale Cutkosky, Alessandro De Stefani, Elo\'{i}sa Grifo, Craig Huneke, Daniel Katz, Linquan Ma, Sarasij Maitra, Shravan Patankar, Karl Schwede, Austyn Simpson, and Kevin Tucker for their time and discussions. The author especially thanks Hanlin Cai, Linquan Ma, Karl Schwede, and Kevin Tucker. 

The author sends a most grateful gesture of appreciation to an anonymous referee for their valuable feedback and suggestions of improvement to the article.

\section{Integral Closure, Valuations, Multiplicity, and the Izumi-Rees Theorem}\label{Section valuations}

Assume that $R$ is an excellent reduced ring, $K$ the total ring of fractions of $R$, and $\overline{R}$ its integral closure in $K$.

\subsection{Integral Closure of Ideals and Huneke's Uniform Theorems}

The integral closure of an ideal $I\subseteq $ is the ideal $\overline{I}$ consisting of elements $x\in R$ that satisfy an equation of the form $x^t+a_1x^{t-1}+\cdots + a_{t-1}x + a_t=0$ so that $a_i\in I^i$ for all $1\leq i\leq t$, see \cite[Chapter~1]{SwansonHuneke} for an introduction on the integral closures of ideals.

Our work relies on the Uniform Artin-Rees and Uniform Brian\c{c}on-Skoda Theorems of \cite{HunekeUniformBounds}. We recall the Uniform Artin-Rees and Brian\c{c}on-Skoda properties here for convenience. 

\begin{definition}
    \label{definition uniform properties}
    Let $R$ be a Noetherian ring. 
    \begin{itemize}
        \item An ideal $J\subseteq R$ has the \emph{Uniform Artin-Rees Property} if there exists a constant $A$ so that for every ideal $I\subseteq R$ and $n\in\mathbb{N}$,
        \[
        J\cap I^{n+A}\subseteq JI^n.
        \]
        The constant $A$ is a \emph{Uniform Artin-Rees bound} of the ideal $J\subseteq R$. The ring $R$ has the \emph{Uniform Artin-Rees Property} if every ideal of $R$ enjoys the Uniform Artin-Rees Property.

        \smallskip
        
        \item The ring $R$ satisfies the \emph{Uniform Brian\c{c}on-Skoda Property} if there is a natural number $B$ such that for all ideals $I\subseteq R$ for all $n\in\mathbb{N}$,
        \[
        \overline{I^{n+B}}\subseteq I^n.
        \]
        The constant $B$ is a \emph{Uniform Brian\c{c}on-Skoda bound} of $R$.
    \end{itemize}
\end{definition}

Huneke conjectured that any excellent Noetherian ring $R$ of finite Krull dimension possesses the Uniform Artin-Rees property \cite[Conjecture~1.3]{HunekeUniformBounds}, and that any excellent Noetherian reduced ring of finite Krull dimension exhibits the Uniform Brian\c{c}on-Skoda Property \cite[Conjecture~1.4]{HunekeUniformBounds}. Strong support for these conjectures arises from the same paper, where they are proven under certain additional mild hypotheses.

\begin{theorem}[{\cite[Uniform Artin-Rees and Brian\c{c}on-Skoda Theorems]{HunekeUniformBounds}}]
    \label{Thm Huneke's Uniform Theorems}
    Let $R$ be a Noetherian ring which is either
    \begin{itemize}
        \item essentially of finite type over a local ring;
        \item of prime characteristic and $F$-finite;
        \item essentially of finite type over $\mathbb{Z}$;
    \end{itemize}
    then $R$ enjoys the Uniform Artin-Rees Property. If in addition to one of the above properties the ring $R$ is excellent and reduced, then $R$ enjoys the Uniform Brain\c{c}on-Skoda Property.
\end{theorem}

\subsection{Rees Valuations} Continue to assume $R$ is an excellent reduced ring and $I\subseteq R$ is an ideal. Introduce a variable $T$, the \emph{extended Rees algebra of $I$} is the $\mathbb{Z}$-graded ring
\[
R[IT, T^{-1}] = \cdots \oplus RT^{-2} \oplus RT^{-1} \oplus R \oplus IT \oplus I^2T^2 \oplus \cdots.
\]
The $n$th degree component of $R[IT,T^{-1}]$ is $I^nT^n$ for $n > 0$, coinciding with $R$ in non-positive degrees. In particular,
\[
\frac{R[IT, T^{-1}]}{T^{-1}R[IT,T^{-1}]} \cong \frac{R}{I}\oplus \frac{I}{I^2}T \oplus \frac{I^2}{I^3}T^2\oplus \cdots 
\]
is the \emph{associated graded ring of $I$} and denoted by $\Gr_I(R)$. If $I=(f_1,f_2,\ldots,f_t)$, then the homogeneous localizations of $R[IT,T^{-1}]$ at the elements $f_iT$ give the Laurent polynomial ring over the affine charts of the blowup:
\[
R[(f_1,f_2,\ldots,f_t)T,T^{-1}]_{f_iT} = R\left[\frac{f_1}{f_i},\frac{f_2}{f_i},\ldots,\frac{f_t}{f_i}\right][T,T^{-1}].
\]

Let $\overline{R[IT, T^{-1}]}$ denote the integral closure of $R[IT, T^{-1}]$ in its total ring of fractions $K(T)$. If $n\geq 1$, we have $T^{-n}\overline{R[IT, T^{-1}]}\cap R = \overline{I^n\overline{R}}\cap R=\overline{I^n}$, and $\overline{R[IT,T^{-1}]}$ coincides with $\overline{R}$ in non-positive degrees. Notably, if $R=\overline{R}$, i.e., $R$ is \emph{normal}, then
\[
\overline{R[IT, T^{-1}]} = \cdots \oplus RT^{-2} \oplus RT^{-1} \oplus R \oplus \overline{I}T \oplus \overline{I^2}T^2 \oplus \cdots.
\]

Homogeneous localizations of $\overline{R[IT,T^{-1}]}$ at the elements $f_iT$ give the Laurent polynomial ring of the affine charts of the normalized blowup:
\[
\overline{R[(f_1,f_2,\ldots,f_t)T,T^{-1}]}_{f_iT} = \overline{R\left[\frac{f_1}{f_i},\frac{f_2}{f_i},\ldots,\frac{f_t}{f_i}\right]}\left[T,T^{-1}\right].
\]

The \emph{Rees valuations of $I$}, denoted $\mathcal{R}_I$, are the discrete valuation rings obtained through homogeneous localization of the associated primes of $T^{-1}\overline{R[IT,T^{-1}]}$.\footnote{The normalized extended Rees algebra $\overline{R[IT,T^{-1}]}$ enjoys Serre's conditions $(S_2)$ and $(R_1)$. As $T^{-1}\overline{R[IT,T^{-1}]}$ is a principal ideal, every associated prime of $T^{-1}\overline{R[IT,T^{-1}]}$ has height $1$ by the $(S_2)$ property, and localizations at such primes are discrete valuation rings by the $(R_1)$ property.} If $\nu\in\mathcal{R}_I$, then we typically denote the corresponding minimal primes of $T^{-1}\overline{R[IT,T^{-1}]}$ as $Q_\nu$. Minimal primes of $T^{-1}\overline{R[IT,T^{-1}]}$ are the \emph{exceptional primes} of $\overline{R[IT,T^{-1}]}$.

\begin{example}
    Let $R=\CC[x_1,x_2,x_3]/(x_1x_2+x_3^3)$ and $\fm=(x_1,x_2,x_3)$. Then $R$ is a normal domain with an isolated rational double point singularity at $\fm$. It is not difficult to show that $\fm^n=\overline{\fm^n}$ for all $n\in\NN$ and therefore
    \[
    R[\fm T,T^{-1}]=\overline{R[\fm T,T^{-1}]}\cong \frac{\CC[y_1,y_2,y_3,T^{-1}]}{(y_1y_2+T^{-1}y_3^3)}.
    \]
    The injective $\CC$-algebra map $R\to R[\fm T,T^{-1}]$ is defined by $x_i\mapsto T^{-1}y_i$ for all $1\leq i\leq 3$. Observe that $T^{-1}\overline{R[\fm T,T^{-1}]}=(T^{-1},y_1)\cap (T^{-1},y_2)$. Therefore $\RR_{\fm}$ is a $2$-element set with Rees valuation rings $\left(\CC[y_1,y_2,y_3]/(y_1y_2+T^{-1}y_3^3)\right)_{(T^{-1},y_1)}$ and $\left(\CC[y_1,y_2,y_3]/(y_1y_2+T^{-1}y_3^3)\right)_{(T^{-1},y_2)}$.
\end{example}

\begin{remark}
The Rees valuations of an ideal $I\subseteq R$ are in bijective correspondence with exceptional components of the normalized blowup of $I$, $\overline{\Bl(I)}\to \Spec(R)$. Algebraic properties of the (normalized) extended Rees algebra of $I$ reflect geometric properties of the (normalized) blowup of $I$ and conversely.
\end{remark}

\begin{remark}
    A \emph{divisorial valuation of $R$} is a discrete valuation $\nu: K\to \mathbb{Z}$, non-negative on $R$, with the additional property that if $\fp_\nu$ the center of $\nu$ in $R$ and $\fm_\nu$ the maximal ideal of $V_\nu$, then $\mbox{tr.deg}_{R_{\fp_\nu}/\fp_\nu R_{\fp_\nu}}\left(V/\fm_\nu\right) = \height(\fp_\nu)-1$. Every Rees valuation of $R$ is a divisorial valuation and conversely, see \cite[Lemma~6.1]{CutkoskySarkar}.
\end{remark}

The following theorem is a list of known properties of Rees valuations used throughout this article.

\begin{theorem}[Properties Rees Valuations]
    \label{theorem standard facts about integral closure and valuations}
    Let $R$ be an excellent Noetherian reduced ring and $I\subseteq R$ an ideal not contained in a minimal prime of $R$. Let $\mathcal{R}_I$ denote the set of Rees valuations of $I$. For each $\nu\in\RR_I$ let $Q_\nu$ be the corresponding exceptional prime of $\overline{R[IT,T^{-1}]}$ and $\nu(I)$ the unique natural number so that
    \[
    T^{-1}\overline{R[IT,T^{-1}]}=\bigcap_{\nu\in\RR_I} Q_\nu^{(\nu(I))}.
    \]
    \begin{enumerate}
        \item\label{Valuation Criteria} Let $f\in R$. Then $f\in \overline{I^n}$ if and only if $\nu(f)\geq n\nu(I)$ for all $\nu\in\mathcal{R}_I$, \cite{ReesValuationsAssociatedToIdeals}.\footnote{This criteria for containment in the integral closure of the power of an ideal is Rees' \emph{valuation criteria.}} Consequently,
        \begin{itemize}
            \item If $W$ is a multiplicative set then the Rees valuations of $IR_W$ are the Rees valuations of $I$ whose centers do not intersect $W$, \cite[Proposition~10.4.1]{SwansonHuneke}.
            \item $\mathcal{R}_I=\mathcal{R}_{\overline{I}}$;
            \item If $t\in\NN$, then $\mathcal{R}_I=\mathcal{R}_{I^t}$.
        \end{itemize}
        \item\label{associated primes stabilize} Let $\fp\in\Spec(R)$. Then $\fp$ is an associated prime of $\overline{I^n}$ for some $n$ if and only if $\fp$ is a center of a Rees valuation of $I$. If $\fp$ is an associated prime of $\overline{I^{n_0}}$ then $\fp$ is an associated prime of $\overline{I^n}$ for all $n\geq n_0$, \cite[Theorem~2.4 and Theorem~2.7]{RatliffAsymptoticPrimesIntegral}.
        
        \item\label{analytic spread criteria for center} If $\fp\in\Spec(R)$ is of height $h$, then $\fp$ is a center of a Rees valuation of $I$ if and only if $IR_\fp$ has analytic spread $h$, \cite[Theorem~3]{McAdamAsymptotic}.
        
        \item\label{affine chart criteria} If $(R,\fm,k)$ is local and $a\in I$ has the property that $\nu(a)=\nu(I)$ for every $\nu\in\mathcal{R}_I$, then $\mathcal{R}_I$ is the set of all valuation domains of the associated primes of the principal ideal generated by $a$ in the affine chart $\overline{R\left[\frac{I}{a}\right]}$ of the normalized blowup of $I$. Such an element $a$ exists if $R$ has infinite residue fields, \cite[Proposition~10.2.5]{SwansonHuneke}. 
        
        \item\label{Minimal reduction criteria} If $(R,\fm,k)$ is local with infinite residue field and $I\subseteq R$ is $\fm$-primary, then $f\in I$ is part of a minimal reduction of $I$ if and only if $\nu(f)=\nu(I)$ for all $\nu\in\mathcal{R}_I$, \cite[Page~437-438]{Sally}.
    \end{enumerate}
\end{theorem}

\subsection{Gaussian Extensions of Valuations}
Suppose that $X$ is a variable and consider $R\to R[X]$. The fraction field of $R[X]$ is $K(X)$. If $\nu$ is a $K$-valuation then $\nu$ extends to a $K(X)$-valuation $\nu'$ via the \emph{Gaussian extension of $\nu$}; if $f\in K[x]$, $f=a_0+a_1x +\cdots +a_nx^n$ with $a_i\in K$, then $\nu'(f)=\min\{\nu(a_0),\nu(a_1),\ldots,\nu(a_n)\}$, see \cite[Remark~6.1.3]{SwansonHuneke}.

\begin{lemma}
    \label{lemma reduction to an infinite field}
    Let $R$ be a reduced excellent ring, $X$ a variable, $R'=R[X]$, $I\subseteq R$ an ideal of $R$ not contained in a minimal prime of $R$, and $I'=IR'$. Then
    \begin{enumerate}
        \item The ring $R$ is normal if and only if $R'$ is normal.
        \item For every $n\in\NN$, $\overline{I^n}R'=\overline{(I')^n}$.
        \item For every $n\in\NN$, $\overline{I^n} = \overline{(I')^n}\cap R$.
    \end{enumerate}    
    Moreover, there is a bijection of Rees valuations $\mathcal{R}_I$ with the Rees valuations $\mathcal{R}_{I'}$, given by Gaussian extension of valuations from $R$ to $R[X],$ with the following properties:
    
    \begin{enumerate}  
    \setcounter{enumi}{3}
        \item If $\nu \in \mathcal{R}_I$ and $\nu'$ the corresponding element of $\mathcal{R}_{I'}$  then $\nu(I)=\nu(I')$.
        \item If $\nu \in \mathcal{R}_I$ and $\nu'$ the corresponding element of $\mathcal{R}_{I'}$ then for all $f\in R$, $\nu(f)=\nu'(f)$. 
        \item If $\fp_\nu$ is the center of $\nu\in\mathcal{R}_I$ then the center of the corresponding Rees valuation $\nu'\in\mathcal{R}_{I'}$ is $\fp_{\nu}R'$.
        \item\label{basically same exceptional prime} If $\nu\in\RR_{\fm}$, $\nu'$ the Gaussian extension of $\nu$ to $R[X]$, $Q_\nu$ and $Q_{\nu'}$ the respective exceptional primes of $\overline{R[IT, T^{-1}]}$ and $\overline{R'[I'T,T^{-1}]}$ respectively, then
        \[
        Q_{\nu'} = Q_{\nu}\overline{R'[I'T,T^{-1}]}.
        \]
        In particular, if $(R,\fm,k)$ is local and $I$ is $\fm$-primary, then
        \[
        e\left(\frac{\overline{R[I T,T^{-1}]}}{Q_\nu}\right) = e\left(\frac{\overline{R'[I'T,T^{-1}]}}{Q_{\nu'}}\right).
        \]
    \end{enumerate}
\end{lemma}

\begin{proof}
     Let $K$ be the fraction field of $R$. Then $K(X)$ is the fraction field of $R'$.  If $\fp$ is a prime ideal of $R$ then $\fp' = \fp R'$ is a prime ideal of $R'$ whose height agrees with the height of $\fp$ and $\fp'\cap R=\fp$. The map $R\to R'$ is faithfully flat with regular fibers, therefore $R$ is normal if and only if $R'$ is normal. If $I\subseteq R$ is an ideal, then it is simple to check that $\overline{I^n}R'=\overline{(I')^n}$ and $\overline{I^n} = \overline{(I')^n}\cap R$.
    
     For each divisorial valuation $\nu$ of $R$ let $\nu'$ be the Gaussian extension of $\nu$ to $R'$. Then the collection of valuations $\{\nu'\mid \nu\in \mathcal{R}_I\}$ will form the Rees valuations of $I'$ and has the described properties of the lemma. Moreover, if $\nu,\nu', Q_\nu,$ and $Q_{\nu'}$ are as in the statement of (\ref{basically same exceptional prime}), then $\overline{R'[I'T,T^{-1}]}\cong \overline{R[I T,T^{-1}]}[X]$ and hence $Q_{\nu'} = Q_{\nu}\overline{R'[I'T,T^{-1}]}$.
\end{proof}

\subsection{Multiplicity and Valuations}

Let $R$ be an excellent reduced ring.  If $M$ is a finite length $R$-module then $\mathcal{L}(M)$ is the length of $R$. If $(R,\fm,k)$ is local, $M$ a non-zero and finitely generated $R$-module of Krull dimension $d$, and $I\subseteq R$ an $\fm$-primary ideal, then the (Hilbert-Samuel) multiplicity of $M$ with respect to $I$ is  $e_I(M):=\lim_{n\to\infty}\frac{d!\mathcal{L}(M/I^nM)}{n^{d}}$. If $(R,\fm,k)$ is local then $e(M)$ is the Hilbert-Samuel multiplicity of $M$ with respect to the maximal ideal. If $R$ is $\ZZ$-graded, $I\subseteq R$ a homogeneous ideal so that $R/I$ is of finite length and an Artin local ring in degree $0$, and $M=\oplus_{n\in \ZZ}M_n$ a finitely generated graded $R$-module of dimension $d$, then the multiplicity of $M$ with respect to $I$ is $e_I(M)=\lim_{n\to \infty}\frac{(d-1)!\ell(M_n/IM\cap M_n)}{n^{d-1}}$. In particular, if $(R,\fm,k)$ is local $I\subseteq R$ an $\fm$-primary ideal, then 
\[
e_I(R)=e_{(IT,T^{-1})}\left(\frac{R[I T,T^{-1}]}{T^{-1}R[IT,T^{-1}]}\right).
\]

If $R$ is excellent and equidimensional then Hilbert-Samuel multiplicity defines an upper semi-continuous function $\Spec(R)\to \NN$ by $\fp\mapsto e(R_\fp)$, \cite[Theorem~4]{Bennett}. We implicitly use this result throughout this article when we assert either of the following consequences of Nagata's criteria for openness, \cite[Theorem~24.2]{Matsumura}, and quasi-compactness of $\Spec(R)$:
\begin{itemize}
    \item If $\fp\subseteq \fq$ are prime ideals then $e(R_\fp)\leq e(R_\fq)$.
    \item There exists an upper bound $e$ for the Hilbert-Samuel multiplicity of each localization of $R$ at a prime ideal, i.e., if $\fp\in\Spec(R)$, then $e(R_\fp)\leq e$.
\end{itemize}

The following theorem is Rees' \emph{Order Ideal Theorem}, which plays a crucial role in the comparison of divisorial valuations with differing centers in this article. We augment The Order Ideal Theorem statement with insights not explicitly stated in Rees' original statements. Instead, the additional insights can be derived from Rees' proof. We provide a streamlined and somewhat novel proof of Rees' result. Doing so eliminates extensive terminology translation and justifications that exceed the following presentation, resulting in a more concise treatment of the necessary materials.

\begin{theorem}[{\cite[Rees' Order Ideal Theorem]{ReesDegree}}]
\label{theorem Rees multiplicity formula}
Let $(R,\fm,k)$ be an equidimensional local ring and of Krull dimension $d$. Suppose that $R$ is analytically reduced and $I$ an $\fm$-primary ideal. For each Rees Valuation $\nu\in\mathcal{R}_I$ let $Q_\nu$ be the corresponding exceptional prime of the normalized extended Rees algebra $\overline{R[IT,T^{-1}]}$. If $f$ is an element of $\fm$ avoiding all minimal primes of $R$, then
\[
e_{\frac{(I,f)}{(f)}}\left(\frac{R}{fR}\right) = \sum_{\nu \in \mathcal{R}_I} \nu(f)e\left(\frac{\overline{R[IT,T^{-1}]}}{Q_\nu}\right).
\]
\end{theorem}
\begin{proof}
    Let $\fa= \bigcap_{\nu \in \mathcal{R}_I}Q_\nu^{(\nu(f))}\subseteq \overline{R[IT,T^{-1}]}$. By the associativity formula for multiplicity, \cite[Theorem~14.7]{Matsumura},
    \[
    e\left(\frac{\overline{R[IT,T^{-1}]}}{\fa}\right) = \sum_{\nu \in \mathcal{R}_I} \nu(f)e\left(\frac{\overline{R[IT,T^{-1}]}}{Q_\nu}\right).
    \]  
    
    The degree $n$ piece of $\fa$, denoted by $\fa_n$, is 
    \begin{align*}
    \fa_n &= (T^{-n}\fa )\cap R \\
    &= \left(\bigcap_{\nu\in\mathcal{R}_I}Q_\nu^{(\nu(I)n + \nu(f))}\right)\cap R \\ 
    & = \{x\in R \mid \nu(x)\geq n\nu(I) + \nu(f), \forall \nu\in\mathcal{R}_I\}.
    \end{align*}
    By Rees' Valuation criteria for containment in integral closure, $\bigcap_{\nu \RR_I}I_{\nu\geq n\nu(I)}=\overline{I^n}$. Therefore $g\in \overline{I^n}$ if and only if for all $\nu\in\RR_I$, $\nu(gf)=\nu(g)+\nu(f)\geq n\nu(I)+\nu(f)$. Hence, $\overline{I^n} = (\fa_n:_Rf)$.

   If $h\geq \nu_i(f)$ for all $1\leq i\leq t$, then $\overline{I^{n+h}}\subseteq \fa_n \subseteq \overline{I^n}$ for all $n$, i.e. the chains of ideals $\{\fa_n\}$, $\{\overline{I^n}\}$ are cofinal. We are assuming $R$ is analytically reduced, therefore $\{\fa_n\}$ is also cofinal with $\{I^n\}$. In conclusion,
    \begin{align*}
    e\left(\frac{\overline{R[IT,T^{-1}]}}{\fa}\right) &= \lim_{n\to \infty}\frac{(d-1)!}{n^{d-1}}\mathcal{L}\left(\frac{\overline{I^n}}{\fa_n}\right) \\
    & =  \lim_{n\to \infty}\frac{(d-1)!}{n^{d-1}}\mathcal{L}\left(\frac{(\fa_n:f)}{\fa_n}\right) \\
    &= \lim_{n\to \infty}\frac{(d-1)!}{n^{d-1}}\left(\mathcal{L}\left(\frac{R}{\fa_n}\right)-\mathcal{L}\left(\frac{R}{(\fa_n:f)}\right)\right)\\
    &= \lim_{n\to \infty}\frac{(d-1)!}{n^{d-1}}\mathcal{L}\left(\frac{R}{(\fa_n,f)}\right) \\
    &= \lim_{n\to \infty}\frac{(d-1)!}{n^{d-1}}\mathcal{L}\left(\frac{R}{(I^n,f)}\right) \\
    &= e_{\frac{(I,f)}{(f)}}\left(\frac{R}{fR}\right).
    \end{align*}
\end{proof}

\begin{corollary}[Corollary of Rees' Order Ideal Theorem]
\label{corollary Rees multiplicity formula}
Let $(R,\fm,k)$ be an equidimensional local ring and of Krull dimension $d$. Suppose that $R$ is analytically reduced and $I$ an $\fm$-primary ideal. For each Rees Valuation $\nu\in\mathcal{R}_I$ let $Q_\nu$ be the corresponding exceptional prime of the normalized extended Rees algebra $\overline{R[IT,T^{-1}]}$. Then
\[
e_I(R) = \sum_{\nu\in\mathcal{R}_I} \nu(I)e\left(\frac{\overline{R[IT,T^{-1}]}}{Q_\nu}\right).
\]
In particular, if $\nu\in\mathcal{R}_I$ then $\nu(I)\leq e_I(R)$.
\end{corollary}

\begin{proof}
    By Lemma~\ref{lemma reduction to an infinite field}, we may assume $R$ has an infinite residue field. Then there exists a parameter element $f \in I$ with the property that $\nu(f) = \nu(I)$ for all $\nu \in \mathcal{R}_I$. In particular, $f$ is a part of a minimal reduction of $I$, and $e_I(R) = e_{\frac{(I,f)}{(f)}}(R/fR)$. The corollary is then an application of Theorem~\ref{theorem Rees multiplicity formula}.
\end{proof}

Theorem~\ref{theorem Rees multiplicity formula} is a generalization of the observation that if $(R,\fm,k)$ is a regular local ring then for all $0\not=f\in \fm$, $e(R/fR)=\ord_\fm(f)$. Indeed, if $R$ is regular then the associated graded ring $R[\fm T,T^{-1}]/T^{-1}R[\fm T,T^{-1}]$ is a polynomial ring over $k$ in $\dim(R)$ variables. It follows that $R[\fm T,T^{-1}]$ is a normal domain and $T^{-1}$ is a prime element. Therefore the collection of Rees valuations of $\fm$ is the $1$-element set $\RR_\fm=\{\omega\}$ and for all $0\not = f\in \fm$, $\omega(f)=\ord_\fm(f)$. By Theorem~\ref{theorem Rees multiplicity formula}, $e(R/fR)=\ord_\fm(f)e\left(\frac{R[\fm T,T^{-1}]}{T^{-1}R[\fm T^{-1}]}\right)=\ord_\fm(f)$.

Regular local rings are not the only class of local rings whose maximal ideal admits a single Rees valuation determined by $\fm$-adic order. The following proposition, likely known by experts, points out that the localization of a standard graded normal domain, at the unique homogeneous maximal ideal, produces a local ring $(R,\fm,k)$ whose maximal ideal admits a single Rees valuation that agrees with $\fm$-adic order.

\begin{proposition}\label{proposition Rees valuation of standard graded normal domain}
    Let $k$ be a field, $S$ a standard graded normal domain over $k$ with homogeneous maximal ideal $\MM$, and let $e(S)$ be the multiplicity of $S$ with respect to the maximal ideal $\MM$. Let $(R,\fm,k)$ be the local Noetherian ring obtained through (non-homogeneous) localization of $S$ with respect to $\MM$. Then the collection of Rees valuations of the maximal ideal of $R$ is a $1$-element set, $\RR_{\fm}=\{\omega\}$, so that for all $0\not = f\in \fm$, $\omega(f)=\ord_\fm(f)$ and $e(R_\fm/fR_\fm)=e(S)\ord_\fm(f)$.
\end{proposition}

\begin{proof}
    The associated graded ring $\Gr_{\fm}(R)$ is isomorphic to the standard graded normal domain $S$. The property of normality deforms by \cite[Proposition I.7.4]{Seydi}. Therefore the extended Rees algebra $R[\fm T, T^{-1}]$ is normal and $T^{-1}R[\fm T, T^{-1}]$ is a prime element. Hence $\RR_{\fm}=\{\omega\}$ is a one element set, $\omega(f)=\ord_\fm(f)$, and by Theorem~\ref{theorem Rees multiplicity formula}, $e(R/fR) = e(S)\omega(f)=e(S)\ord_\fm(f)$ for all $0\not = f\in \fm$. 
\end{proof}

\subsection{Rees' Order Ideal Theorem and the Izumi-Rees Theorem}
Let $R$ be an excellent normal domain, $\fp\in\Spec(R)$, and $0\not=f\in\fp$. Intersection properties of exceptional components of normalized blowups, described by the below stated Izumi-Rees Theorem, will be utilized in parallel with semi-continuity of multiplicity, Rees Order Ideal Theorem, and the Uniform Brain\c{c}on-Skoda Theorem in Section~\ref{Section Uniform Izumi} to compare the values of $e(R_\fp/fR_\fp)$ and $\ord_\fp(f)$.

\begin{theorem}[{\cite[Izumi-Rees Theorem]{Izumi, ReesIzumisTheorem, HublSwanson}}]
    \label{Standard Izumi-Rees theorem}
    Let $R$ be an excellent Noetherian normal domain and $\fp\subseteq R$ a prime ideal. If $\RR=\{\nu_1,\nu_2,\ldots,\nu_t\}$ are divisorial valuations of $R$ centered on  $\fp$, e.g. $\RR$ is the collection of Rees valuations of the maximal ideal of $R_\fp$, then there is a constant $E$, depending on the collection of divisorial valuations $\RR$, such that for all $f\in R$, for all $1\leq i,j,\leq t$
    \[
    \nu_i(f)\leq E\nu_j(f).
    \]
    The constant $E$ is an \textbf{Izumi-Rees bound} of the collection of valuations $\RR$.
\end{theorem}

\begin{remark}
    The Izumi-Rees Theorem, as presented in \cite{HublSwanson}, is a strengthening of the Izumi-Rees Theorem presented in \cite{ReesIzumisTheorem}: If $R$ enjoys the hypotheses of Theorem~\ref{Standard Izumi-Rees theorem}, $\nu$ a Rees valuation of $R$ centered on a prime ideal $\fp\in \Spec(R)$, then there exists a constant $C$, depending on $\nu$, so that for all Rees valuations $\omega$ centered on $\fp$ and all $f\in R$, $\nu(f)\leq C \omega(f)$.
\end{remark}

\begin{proposition}
    \label{proposition of rees order ideal thm and izumi rees}
    Let $(R,\fm,k)$ be an excellent local normal domain and $\RR_{\fm}$ the collection of Rees valuations of $\fm$. Suppose that $B\in\NN$ is so that $\overline{\fm^{n+B}}\subseteq \fm^n$ for all $n\in\NN$.
    \begin{enumerate}
        \item\label{multiplicity implies  izumi} If $C\in\NN$ is so that for all $0\not=f\in \fm$, $e_\fm(R/fR)\leq C\ord_\fm(f)$, then for all Rees valuations $\nu_1,\nu_2\in \RR_\fm$ and $0\not=f\in \fm$,
        \[
        \nu_1(f)\leq (C-1)\nu_2(f).
        \]
        \item\label{izumi implies multiplicity} If $E\in \NN$ is so that for all $\nu_1,\nu_2\in \RR_\fm$ and $0\not = g\in\fm$ that $\nu_1(g)\leq E\nu_2(f)$, then for all $0\not =f\in \fm$,
        \[
        e_\fm(R/fR)\leq 2BEe_\fm(R)^2\ord_\fm(f).
        \]
    \end{enumerate}
\end{proposition}

\begin{proof}
    Suppose that $C$ is a constant so that if $0\not=f\in\fm$ then $e_\fm(R/fR)\leq C\ord_\fm(f)$ and let $\nu_1,\nu_2\in\RR_\fm$ be Rees valuations of the maximal ideal of $R$. There exists $t\geq 1$ so that 
    \[
    t\nu_2(\fm)\leq \nu_2(f)< (t+1)\nu_2(\fm).
    \]
    Then $f\not\in \overline{\fm^{t+1}}$ and hence $\ord_\fm(f)\leq t$. By Rees' Order Ideal Theorem and by assumption,
    \[
    \nu_1(f)+\nu_2(f)\leq e_\fm(R/fR)\leq Ct\leq \frac{C\nu_2(f)}{\nu_2(\fm)}.
    \]
    In particular,
    \[
    \nu_1(f)\leq (C-1)\nu_2(f).
    \]

    Conversely, suppose $E\in\NN$ has the property that for all $\nu_1,\nu_2\in\RR_\fm$ and $0\not=f\in R$ that $\nu_1(f)\leq E\nu_2(f)$. There exists a Rees valuation $\omega\in \RR_\fm$ so that 
    \begin{align*}
    \omega(f)&\leq (B+\ord_\fm(f))\omega(\fm){\footnotesize \, (\mbox{because } \overline{\fm^{\ord_\fm(f)+B}}\subseteq \fm^{\ord_\fm(f)})}\\
    &\leq 2B\ord_\fm(f)\omega(\fm) \\
    &\leq 2B\ord_\fm(f)e_\fm(R) {\footnotesize \, (\mbox{by Corollary~\ref{corollary Rees multiplicity formula}})}.
    \end{align*}
    By assumption, if $\nu\in\RR_\fm$ then $\nu(f)\leq E\omega(f)\leq 2BEe_\fm(R)\ord_\fm(f)$. Therefore
    \begin{align*}
    e_\fm(R/fR)&=\sum_{\nu\in\RR_\fm}\nu(f)e\left(\frac{\overline{R[\fm T,T^{-1}]}}{Q_\nu}\right){\footnotesize{(\mbox{by Theorem~\ref{theorem Rees multiplicity formula}}})}\\
    &\leq 2BEe_\fm(R)\ord_\fm(f)\left(\sum_{\nu\in\RR_\fm}e\left(\frac{\overline{R[\fm T,T^{-1}]}}{Q_\nu}\right)\right)\\
    &\leq 2BEe_\fm(R)^2\ord_\fm(f) {\footnotesize \, (\mbox{by Corollary~\ref{corollary Rees multiplicity formula}})}.
    \end{align*}
\end{proof}

  Proposition~\ref{proposition of rees order ideal thm and izumi rees} and the Uniform Brian\c{c}on-Skoda property provide an equivalent characterization of the Uniform Izumi-Rees Property introduced in the statement of Main Theorem~\ref{Main Theorem Uniform Izumi Rees}.

\begin{corollary}
    \label{corollary multiple definitions of Uniform Izumi-Rees Property}
    Let $R$ be an excellent Noetherian domain that enjoys the Uniform Brian\c{c}on-Skoda Property. Then the following are equivalent. 
    \begin{enumerate}
        \item There exists constant $E$ so that for every prime ideal $\fp\in\Spec(R)$, for all Rees valuations $\nu_1,\nu_2\in\mathcal{R}_{\fp R_\fp}$ of the maximal ideal of $R_\fp$, and for all $0\not = f\in R$,
        \[
        \nu_1(f)\leq E\nu_2(f).
        \]
        \item There exists a constant $C$ so that for every prime ideal $\fp\in\Spec(R)$ and for all $0\not=f\in \fp$, 
        \[
        e(R_\fp/fR_\fp)\leq C\ord_\fp(f).
        \]
    \end{enumerate}
\end{corollary}

\subsection{Equimultiplicity} In proofs to come, we require a comparison of multiplicities of the form $e(R_\fp)$ and $e_{\fq}(R_\fq/xR_\fq)$ where $\fp\subseteq \fq\in\Spec(R)$ and $\fq R_\fq = (\fp,x)R_\fq$. Central to our comparisons of multiplicity is the notion of analytic spread. If $(R,\fm,k)$ is a local ring and $I\subseteq R$ is an ideal, then the \emph{analytic spread} of $I$ is the Krull dimension of the standard graded $k$-algebra $\Gr_{I}(R)\otimes_R R/\fm \cong \bigoplus_{n\geq 0}\frac{I^n}{\fm I^{n+1}}$ and denoted by $\ell(I)$. The following Theorem~\ref{theorem equimultiplicity as presented by Ilya} is an application of the theory of equimultiple ideals.  Theorem~\ref{theorem equimultiplicity as presented by Ilya} and the lemma that follows provide sufficient conditions for $e(R_\fp) = e_{\fq}(R_\fq/xR_\fq)$ whenever $\fp\subseteq \fq$ are prime ideals so that $\fq R_{\fq} = (\fp,x)R_{\fq}$ for some $x\in R$.

\begin{theorem}[{\cite[Theorem~4]{LipmanEquimultiplicity}}]
    \label{theorem equimultiplicity as presented by Ilya}
    Let $(R,\fm,k)$ be a formally equidimensional local ring of Krull dimension $d\geq 2$ and $I\subseteq R$ an ideal of height $d-1$. The following are equivalent.
    \begin{itemize}
        \item $\ell(I) = d-1$;
        \item For some parameter element $x\in \fm$ of $R/I$, $e_{(I,x)}(R) = \sum_{\fp\in\min(I)}e_{(x)}(R/\fp)e_{I}(R_\fp)$;
        \item For every parameter element $x\in \fm$ of $R/I$, $e_{(I,x)}(R) = \sum_{\fp\in\min(I)}e_{(x)}(R/\fp)e_{I}(R_\fp)$.
    \end{itemize}
\end{theorem}

\begin{lemma}
    \label{lemma exceptional components coming from an affine of a nonzero divisor of associated graded ring}
    Let $R$ be an excellent Noetherian normal domain, $(I,x)\subseteq R$ an ideal with the property that there exists $n_0\in\NN$ so that $((I,x)^n:_Rx)= (I,x)^{n-1}$ for all $n\geq n_0$. Let $\mathcal{R}_{(I,x)}$ denote the set of Rees valuations of $(I,x)$.
    \begin{itemize}
        \item For all $n\in\NN$, $(\overline{(I,x)^n}:_Rx)=\overline{(I,x)^{n-1}}$.
        \item For all $\nu\in \mathcal{R}_{(I,x)}$, $\nu(x)=\nu((I,x))$. 
        \item $\mathcal{R}_{(I,x)}$ is the set of all valuation domains of the associated primes of the principal ideal generated by $x$ in the affine chart $\overline{R\left[\frac{(I,x)}{x}\right]}$ of the normalized blowup of $(I,x)$.
    \end{itemize}
\end{lemma}

\begin{remark}
    \label{remark when hypothesis of lemma is met}
    Let $R$ be an excellent Noetherian normal domain, $I\subseteq R$ an ideal, and $x\in R$ an element so that $(I,x)$ is a proper ideal. If there exists $n_0\in\NN$ so that $x$ avoids all primes of $\bigcup_{n\geq n_0}\Ass(R/I^n)$ then $(I,x)$ satisfies the hypothesis of Lemma~\ref{lemma exceptional components coming from an affine of a nonzero divisor of associated graded ring}. First note that $((I,x)^n:_Rx)\supseteq (I,x)^{n-1}$ without any assumptions. If $n\geq n_0$ and $xr\in(I,x)^n=(I^n,x(I,x)^{n-1})$ then there exist $g\in (I,x)^{n-1}$ so that $xr - xg\in I^n$. Hence $r-g\in (I^n:_Rx)=I^n$ and therefore $r\in (I^n,(I,x)^{n-1})=(I,x)^{n-1}$.
\end{remark}

\begin{proof}
    The containment $(\overline{(I,x)^n}:_Rx)\supseteq \overline{(I,x)^{n-1}}$ is an elementary containment property of ideals and is true without assuming $((I,x)^n:_Rx)= (I,x)^{n-1}$. Suppose that $r\in (\overline{(I,x)^n}:_Rx)$, i.e., $xr\in \overline{(I,x)^n}$. By \cite[Corollary~6.8.12]{SwansonHuneke}, there exists $0\not= c \in R$ so that $c(xr)^t\in (I,x)^{nt}$ for all $t\gg 0$. Therefore $cr^t\in ((I,x)^{nt}:_Rx^t)=(I,x)^{(n-1)t}$ for all $t\gg 0$. Hence $r\in \overline{(I,x)^{n-1}}$ by a second application of \cite[Corollary~6.8.12]{SwansonHuneke}.

    Consider the associated graded ring of the normalized extended Rees algebra of $(I,x)$,
    \[
    \overline{\Gr}_{(I,x)}(R):=\frac{\overline{R[(I,x)T,T^{-1}]}}{T^{-1}\overline{R[(I,x)T,T^{-1}]}} = \bigoplus_{n\geq 0}\frac{\overline{(I,x)^n}}{\overline{(I,x)^{n+1}}} T^n.
    \]
    The equality of ideals $(\overline{(I,x)^n}:_Rx)=\overline{(I,x)^{n-1}}$ implies that the degree $1$ element $xT$ is a nonzero divisor of $\overline{\Gr}_{(I,x)}(R)$ . Equivalently, the degree $1$ element $xT$ of the normalized extended Rees algebra $\overline{R[(I,x)T,T^{-1}]}$ avoids all exceptional primes of $\overline{R[(I,x)T,T^{-1}]}$. If $W$ is the complement of the union of the exceptional primes of $\overline{R[(I,x)T,T^{-1}]}$, then $xT$ belongs to $W$ and there are maps of homogeneous localizations
    \[
    \overline{R[(I,x)T,T^{-1}]}_{xT}\cong \overline{R\left[\frac{(I,x)}{(x)}\right]}[T,T^{-1}] \to \overline{R[(I,x)T,T^{-1}]}_W.
    \]
    Even further, $x = xTT^{-1}$ and so $\nu(x) = \nu(xT)+\nu(T^{-1})=0+\nu((I,x))$.
    Therefore $\nu(x)=\nu((I,x))$ for all $\nu\in\mathcal{R}_{(I,x)}$. The third claim of the lemma is an application of Theorem~\ref{theorem standard facts about integral closure and valuations} part (\ref{affine chart criteria}).
\end{proof}

\begin{theorem}[An Application of Equimultiplicity Theory]
    \label{equimultiplicity theorem}
    Let $R$ be an excellent Noetherian normal domain and $\fp\subsetneq \fq\in\Spec(R)$. Assume that $\fq R_\fq = (\fp, x)R_\fp$. The following are equivalent.
    \begin{enumerate}
        \item $\fq R_{\fq} \not \in \bigcup_{n\geq 1}\Ass_R(R_\fq/\overline{\fp^n}R_\fq)$;
        \item $\ell(\fp R_\fq)=\height(\fp)$;
        \item $e(R_\fp) = e\left(R_\fq\right)$.
    \end{enumerate}
    Moreover, if there exists an $n_0$ so that $\fq\not \in \bigcup_{n\geq n_0}\Ass_R(R_\fq/\fp^nR_\fq)$. Then 
    \[
    e(R_\fp)= e\left(R_\fq\right) = e\left(\frac{R_\fq}{xR_\fq}\right).
    \]
\end{theorem}

\begin{proof}
    Equivalence of $(a)$ and $(b)$ can be derived from \cite[Theorem~2.6]{Rees1981}. For a direct presentation, recall that $\fq \in \bigcup_{n\geq 1}\Ass_R(R_\fq/\overline{\fp^n}R_\fq)$ if and only if $\fq$ is a center of a Rees valuation of $\fp$, see Theorem~\ref{theorem standard facts about integral closure and valuations} (\ref{associated primes stabilize}), if and only if $\ell(\fp R_{\fq}) = \height(\fq)$ by \cite[Theorem~10.4.2]{SwansonHuneke}. Moreover, $\height(\fq) = \height(\fp) + 1$ as $\fq R_{\fq} = (\fp,x)R_{\fq}$. Note that $R_{\fq}/\fp R_{\fq}$ is a discrete valuation ring whose maximal ideal is generated by the image of $x$. Therefore equivalence of $(a)$, $(b)$, and $(c)$ is an application of Theorem~\ref{theorem equimultiplicity as presented by Ilya}. 
    
    If $\fq R_\fq$ is not among the elements of $\bigcup_{n\geq n_0}\Ass_R(R_\fq/\fp^nR_\fq)$ then $\fq R_{\fq}\not \in \bigcup_{n\geq 1}\Ass_R(R_\fq/\overline{\fp^n}R_\fq)$, see Theorem~\ref{theorem standard facts about integral closure and valuations} (\ref{associated primes stabilize}), and hence $e(R_\fp)=e(R_\fq)$ by the above. By Lemma~\ref{lemma exceptional components coming from an affine of a nonzero divisor of associated graded ring} and Remark~\ref{remark when hypothesis of lemma is met}, if $\nu\in\RR_{\fq R_\fq}$ then $\nu(x)=\nu(\fq)$. By Theorem~\ref{theorem Rees multiplicity formula} and Corollary~\ref{corollary Rees multiplicity formula}, $e(R_\fq) = e(R_\fq/xR_\fq)$.
\end{proof}

The following specific corollary of Theorem~\ref{equimultiplicity theorem} is used in the proof of Theorem~\ref{theorem uniform izumi arithmetically normal}.

   \begin{corollary}
       \label{corollary extended rees algebra multiplicity computation}
      Let $k$ be a field, and $R=k[x_1,\ldots,x_n]/P$ an affine normal domain, where $P$ is a prime ideal. Endow the polynomial ring $k[x_1,\ldots,x_n]$ with the standard grading (we are not assuming $R$ is graded ring). Consider $f_t,\ldots,f_{t+c}\in k[x_1,\ldots,x_n]$ such that $f_i$ is either $0$ or homogeneous of degree $i$. Assume that the image of $f:=f_t+f_{t+1}+\cdots+f_{t+c}$ in $R$ is nonzero. Let $T^{-1}$ be a variable of degree $-1$ and $k[x_1,\ldots,x_n,T^{-1}]$, a $\ZZ$-graded polynomial ring.

      Define $f'=f_t+T^{-1}f_{t+1}+\cdots + T^{-c}f_{t+c}$ in the $\ZZ$-graded polynomial ring $k[x_1,\ldots,x_n,T^{-1}]$. By changing the expansion of $f'$ around $T^{-1}=0$ to an expansion around $T^{-1}=1$ in the polynomial ring $k[x_1,\ldots,x_n,T^{-1}]$, we obtain polynomial functions $g_t,g_{t+1},\ldots,g_{t+c}\in k[x_1,\ldots,x_n]$ so that
      \[
      f' = (f_t+ f_{t+1}+\cdots+f_{t+c}) +(T^{-1}-1)g_{t}+\cdots + (T^{-1}-1)^cg_{t+c}.
      \]
      Abuse notation and let $f_i, g_i$ denote the images of $f_i,g_i$ in the quotient ring $R$ for each $t\leq i\leq t+c$. For any Rees valuation $\nu\in\RR_{\fm}$ of the maximal ideal $\fm=(x_1,\ldots,x_n)$ of $R$, 
      \[
      \min\{\nu(f_t(\underline{x})+f_{t+1}(\underline{x})+\cdots + f_{t+c}(\underline{x})), \nu(g_t(\underline{x})),\ldots,\nu(g_{t+c}(\underline{x}))\} = \nu(f_t(\underline{x})+f_{t+1}(\underline{x})+\cdots + f_{t+c}(\underline{x})).
      \]
   \end{corollary}

   \begin{proof}
       Consider the extended Rees algebra $R[\fm T,T^{-1}]$. Then $R[\fm T,T^{-1}]$ is a $\ZZ$-graded algebra. Let $y_1,\ldots,y_n$ denote degree $1$-generators of the homogeneous ideal $(\fm T)\subseteq R[\fm T, T^{-1}]$. Then $R[\fm T,T^{-1}]$ is the homomorphic image of the $\ZZ$-graded polynomial ring $k[y_1,\ldots,y_n,T^{-1}]$.  Observe that $(\fm T)\subseteq R[\fm T,T^{-1}]$ is a prime ideal as $R[\fm T,T^{-1}]/(\fm T)\cong k[T^{-1}]$. Moreover,  $R[\fm T,T^{-1}]_{\fm T}\cong R_\fm(T)$. Hence there is a bijection between the Rees valuations of $\fm\subseteq R$ and $\fm T\subseteq R[\fm T,T^{-1}]$ centered on $\fm T$ given by Gaussian extension, see Lemma~\ref{lemma reduction to an infinite field}. 
       
       Let $f'(\underline{y})$ be the homogeneous degree $t$ element of $k[y_1,\ldots,y_n,T^{-1}]$ obtained by substituting $y_i$ for $x_i$. For each $n\in\NN$ the ideal $((\fm T)^n,f')\subseteq R[\fm T,T^{-1}]$ is homogeneous and admits a homogeneous primary decomposition. Therefore the non-homogeneous element $T^{-1}-1$ avoids such components. By Theorem~\ref{equimultiplicity theorem},
       \[
       e\left(\left(\frac{R[\fm T,T^{-1}]}{(f'(\underline{y}))}\right)_{\fm T}\right)=e\left(\left(\frac{R[\fm T,T^{-1}]}{(f'(\underline{y}),T^{-1}-1)}\right)_{(\fm T,T^{-1}-1)}\right).
       \]
       Observe that 
       \[
       \frac{R[\fm T,T^{-1}]}{(f'(\underline{y}),T^{-1}-1)}\cong \frac{R}{(f_t+f_{t+1}+\cdots+f_{t+c})R}.
       \]
       Therefore
       \begin{align}
       \label{equation to reference multiplicity 1}
       e\left(\left(\frac{R[\fm T,T^{-1}]}{(f'(\underline{y}))}\right)_{\fm T}\right)&=e\left(\left(\frac{R[\fm T,T^{-1}]}{(f'(\underline{y}),T^{-1}-1)}\right)_{(\fm T,T^{-1}-1)}\right)\\
       \label{equation to reference multiplicity 2}
       &= e\left(\left(\frac{R}{(f_t+f_{t+1}+\cdots+f_{t+c})R}\right)_{\fm}\right).
       \end{align}
       
       As an element of the polynomial ring $k[y_1,\ldots,y_n,T^{-1}]$,
       \[
       f'(\underline{y})=(f_t(\underline{y})+f_{t+1}(\underline{y})+\cdots + f_{t+c}(\underline{y})) + (T^{-1}-1)g_t(\underline{y})+\cdots+(T^{-1}-1)^cg_{t+c}(\underline{y}).
       \]
       The Rees valuations of $\RR_{(\fm T)_{(\fm T,T^{-1}-1)}}$ are the Gaussian extensions of the Rees valuations of $\RR_\fm$. If $\nu'\in \RR_{(\fm T)_{(\fm T,T^{-1}-1)}}$ we let $\nu\in \RR_\fm$ the corresponding Rees valuation of $\fm\subseteq R$. For every $\nu'\in \RR_{(\fm T)_{(\fm T,T^{-1}-1)}}$
       \begin{align}\label{equation gaussian valuation}
       \nu'(f') = \min\{\nu(f_t+f_{t+1}+\cdots + f_{t+c}), \nu(g_t),\ldots,\nu(g_{t+c})\}.
       \end{align}
       
       By Theorem~\ref{theorem Rees multiplicity formula}, for each $\nu'\in \RR_{(\fm T)_{(\fm T,T^{-1}-1)}}$ there are natural numbers $d_{\nu'}\geq 1$ so that
       \[
       e\left(\left(\frac{R[\fm T,T^{-1}]}{(f')}\right)_{\fm T}\right) = \sum_{\nu'\in \RR_{(\fm T)_{(\fm T,T^{-1}-1)}}}\nu'(f')d_{\nu'}
       \]
       and
       \[
       e\left(\left(\frac{R[\fm T,T^{-1}]}{(f_t+f_{t+1}+\cdots+f_{t+c})}\right)_{\fm T}\right) = \sum_{\nu'\in \RR_{(\fm T)_{(\fm T,T^{-1}-1)}}}\nu(f_t+f_{t+1}+\cdots+f_{t+c})d_{\nu'}.
       \]
       By (\ref{equation to reference multiplicity 1}) and (\ref{equation to reference multiplicity 2}),
       \begin{align}\label{equation gaussian valuation rees order}
       \sum_{\nu'\in \RR_{(\fm T)_{(\fm T,T^{-1}-1)}}}\nu'(f')d_{\nu'} = \sum_{\nu'\in \RR_{(\fm T)_{(\fm T,T^{-1}-1)}}}\nu(f_t+f_{t+1}+\cdots+f_{t+c})d_{\nu'}.
       \end{align}
       By (\ref{equation gaussian valuation}) and (\ref{equation gaussian valuation rees order}), 
       \begin{align*}
           \min\{\nu(f_t(\underline{y})+f_{t+1}(\underline{y})+\cdots + f_{t+c}(\underline{y})), \nu(g_t(\underline{y})),\ldots,\nu(g_{t+c}(\underline{y}))\} = \nu(f_t+f_{t+1}+\cdots+f_{t+c}).
       \end{align*}       
   \end{proof}

\subsection{Homogenization and Projective Closures}\label{subsection homogenization}
    Let $k$ be a field and $R=\frac{k[x_1,x_2\ldots,x_n]}{P}$ an affine domain, $W$ a multiplicative set, and $R_W=\left(\frac{k[x_1,x_2\ldots,x_n]}{P}\right)_W$. The homogenization of $R$, with respect to the presentation $R = \frac{k[x_1,x_2,\ldots,x_n]}{P}$, produces a standard graded domain $S$ that is the coordinate ring of a choice of projective closure of $\Spec(R_W)$. The graded ring $S$ can be derived as follows: let $X_0,X_1,\ldots,X_n$ be variables, $x_i=\frac{X_i}{X_0}$, 
    \[
    \prescript{h}{}{P}=\left(F\in k[X_0,X_1,\ldots,X_n]\mid F\mbox{ is homogeneous and } \frac{F}{X_0^{\deg(F)}}\in P\right),
    \]
    and $S = k[X_0,X_1,X_2,\ldots,X_n]/\prescript{h}{}{P}$. The homogeneous maximal ideal $\mathcal{M}=(X_0,X_1,\ldots,X_n)$ is the irrelevant homogeneous ideal of $S$. Given a homogeneous element $F\in S$, we can lift $F$ to homogeneous element of the polynomial ring $k[X_0,X_1,\ldots,X_n]$ and define $\prescript{a}{}{F}$ to be the image of $\frac{F}{X_0^{\deg(F)}}$ in $R$. The operation $\prescript{a}{}{-}$ is well-defined as $R$ is a domain. Given an ideal $I\subseteq R$, $\prescript{h}{}{I}$ is the ideal generated homogeneous elements $F$ such that $\prescript{a}{}{F}\in I$. The operation $\prescript{h}{}{-}$ is not well-defined on elements of $R$, only the ideals of $R$. The affine variety $\Spec(R)$ is the open subset of the affine piece $D(X_0)$ of $\Proj_k(S)$.

    If $R_W$ is normal, normalization commutes with localization, allowing us to adjust the presentation of $R$ such that $R$ is also normal. Moreover, by embedding $\overline{\Proj_k(S)}$ into projective space, we can refine the presentation of $R$ further to assume that the homogenization $S$ of $R$ is a standard graded normal domain. Consequently, every normal domain essentially of finite type over a field can be viewed as an open subset of an arithmetically normal projective variety.

    The following proposition details points out a well-known relationship between an affine domain $R = k[x_1,x_2,\ldots,x_n]/P$ and the homogeneous ring $S = k[X_0,X_1,X_2,\ldots,X_n]/\prescript{h}{}{P}$.

    \begin{proposition}
        \label{proposition: homogeneization properties wrt localization}
        Let $k$ be a field, $R = k[x_1,x_2,\ldots,x_n]/P$ an affine domain, $X_0,X_1,\ldots,X_n$ homogeneous variables of degree $1$ so that $\frac{X_i}{X_0}= x_i$, and $S$ the standard graded domain $k[X_0,X_1,\ldots,X_n]/\prescript{h}{}{P}$. Abuse notation by letting $X_i$ denote the image $X_i$ in $S$. Then as subsets of the fraction field of $S$,
        \[
        S\left[\frac{X_1}{X_0},\ldots,\frac{X_n}{X_0}\right] \cong R[X_0].
        \]
    \end{proposition}
    \begin{proof}
        The algebra $S\left[\frac{X_1}{X_0},\ldots,\frac{X_n}{X_0}\right]$ is an affine chart of the blowup of the homogeneous maximal ideal of $S$. Therefore $\dim\left(S\left[\frac{X_1}{X_0},\ldots,\frac{X_n}{X_0}\right]\right) = \dim(R)+1$. Moreover, $S\left[\frac{X_1}{X_0},\ldots,\frac{X_n}{X_0}\right]$ is the homomorphic image of the polynomial algebra $k\left[X_0,\frac{X_1}{X_0},\ldots, \frac{X_n}{X_0}\right]$ and the kernel $\fp$ of $k\left[X_0,\frac{X_1}{X_0},\ldots, \frac{X_n}{X_0}\right]\to S\left[\frac{X_1}{X_0},\ldots,\frac{X_n}{X_0}\right]$ contains the extension $\prescript{h}{}{P}$ to $k\left[X_0,\frac{X_1}{X_0},\ldots, \frac{X_n}{X_0}\right]$. The element $X_0$ is a nonzero divisor of $S\left[\frac{X_1}{X_0},\ldots, \frac{X_n}{X_0}\right]$, therefore 
        \[
        \left(\prescript{h}{}{P}k\left[X_0,\frac{X_1}{X_0},\ldots, \frac{X_n}{X_0}\right]:_{k\left[X_0,\frac{X_1}{X_0},\ldots, \frac{X_n}{X_0}\right]} X_0^{\infty}\right)\subseteq \fp.
        \]
        Recall that $\prescript{h}{}{P} = \left(\prescript{h}{}{f}\mid f\in P\right)$. Therefore $Pk\left[X_0,\frac{X_1}{X_0},\ldots, \frac{X_n}{X_0}\right]\subseteq \fp$. But $Pk\left[X_0,\frac{X_1}{X_0},\ldots, \frac{X_n}{X_0}\right]$ is a prime so that $\dim\left(\frac{k\left[X_0,\frac{X_1}{X_0},\ldots, \frac{X_n}{X_0}\right]}{Pk\left[X_0,\frac{X_1}{X_0},\ldots, \frac{X_n}{X_0}\right]}\right) =\dim(R)+1$. Therefore $\fp = Pk\left[X_0,\frac{X_1}{X_0},\ldots, \frac{X_n}{X_0}\right]$ and $S\left[\frac{X_1}{X_0},\ldots,\frac{X_n}{X_0}\right] \cong R[X_0]$ as claimed.        
    \end{proof}

\section{The Uniform Izumi-Rees Property}
\label{Section Uniform Izumi}
This section presents the proof of Main Theorem~\ref{Main Theorem eft over an algebraically closed field} and Main Theorem~\ref{Main Theorem Uniform Izumi Rees}. The below definition is well-defined by Corollary~\ref{corollary multiple definitions of Uniform Izumi-Rees Property}.

\begin{definition}
    \label{definition uniform izumi rees}
    Let $R$ be an excellent Noetherian domain that enjoys the Uniform Brian\c{c}on-Skoda Property. We say that $R$ enjoys the \emph{Uniform Izumi-Rees Property} if the following equivalent properties are enjoyed by $R$.
    \begin{enumerate}
        \item (Valuation criteria of the Uniform Izumi-Rees Property) There exists constant $E$ so that for every prime ideal $\fp\in\Spec(R)$, for all Rees valuations $\nu_1,\nu_2\in\mathcal{R}_{\fp R_\fp}$ of the maximal ideal of $R_\fp$, and for all $0\not = f\in R$,
        \[
        \nu_1(f)\leq E\nu_2(f).
        \]
        The constant $E$ is a \emph{Uniform Izumi-Rees Bound of $R$}.
        \item (Multiplicity criteria of the Uniform Izumi-Rees Property) There exists a constant $C$ so that for every prime ideal $\fp\in\Spec(R)$ and for all $0\not=f\in \fp$, 
        \[
        e(R_\fp/fR_\fp)\leq C\ord_\fp(f).
        \]
        
    \end{enumerate}

\end{definition}

\begin{remark}
    The Uniform Izumi-Rees Property is most naturally studied in a normal domain. Indeed, the Uniform Rees Property cannot be enjoyed by an excellent Noetherian domain $R$ that admits a prime $\fp\in\Spec(R)$ such that the normalization of the localization $R_\fp$ exhibits branching at the maximal ideal. Suppose $\fp\in \Spec(R)$ is a prime of height $h$, $s\geq 2$, and $\fq_1,\ldots,\fq_s\in\Spec(\overline{R})$ are primes of height $h$ lying over $\fp$. Choose an $f\in\fq_1$ that avoids $\cup_{i=2}^s\fq_i$. Let $0 \neq x\in\Ann_R(\overline{R}/R)$ be an element of the conductor. Then for all $t\in\NN$, the element $xf^t\in R$, and $\nu_i(xf^t)=\nu_i(x)$ for all $2\leq i\leq s$. By Theorem~\ref{theorem standard facts about integral closure and valuations} (\ref{Valuation Criteria}), there exists an $n_0$ such that $\ord_\fm(xf^t)\leq n_0$ for all $t$. By Theorem~\ref{theorem Rees multiplicity formula}, $e(R_\fm/xf^tR_\fm)\geq \nu_1(xf^t)\geq t$. 
\end{remark}

\subsection{Quasi-projective varieties over an algebraically closed field}

Main Theorem~\ref{Main Theorem eft over an algebraically closed field} is Corollary~\ref{corollary uniform izumi rees arithmetically normal} of the following theorem.

\begin{theorem}
    \label{theorem uniform izumi arithmetically normal}
    Let $k$ be an algebraically closed field and $R$ a normal domain essentially of finite type over $k$. Let $S$ be a coordinate ring of an arithmetically normal projective closure of $\Spec(R)$ and $\MM$ the homogeneous maximal ideal of $S$. Then for all prime ideals $\fp\in\Spec(R)$ and all $0\not=f\in\fp$,
    \[
    e\left(\frac{R_\fp}{fR_\fp}\right)\leq e(S_{\MM})\ord_\fp(f).
    \]
\end{theorem}

\begin{proof}
    Suppose that $R\cong \left(k[x_1,\ldots,x_n]/P\right)_W$, $X_0,X_1,\ldots, X_n$ variables so that $x_i=\frac{X_i}{X_0}$, and $S=\frac{k[X_0,X_1,\ldots,X_n]}{\prescript{h}{}{P}}$ is a standard graded normal domain with homogeneous maximal ideal $\MM=(X_0,X_1,\ldots,X_n)$. We may assume that $R$ is the affine normal domain $k[x_1,\ldots,x_n]/P$. Let $\fp \in \Spec(R)$. We abuse notation by letting $X_0,X_1,\ldots,X_n$ denote the images of the graded variables in $S$ and $x_1,\ldots,x_n$ the images of the variables in $R$. By \cite[Theorem]{EisenbudHochster} there is an open subset of maximal ideals $U\subseteq V(\fp)\cap \Max(R)$ so that
    \[
    \fp^{(\ord_{\fp}(f))} = \bigcap_{\fm \in U} \fm^{\ord_{\fp}(f)}
    \]
    and a dense open subset of maximal ideals $V\subseteq V(\fp)\cap \Max(R)$ so that
    \[
    \fp^{(\ord_{\fp}(f))} = \bigcap_{\fm \in V} \fm^{\ord_{\fp}(f)}+1.
    \]
    Therefore there exists a maximal ideal $\fm\in \Max(R)$ so that $\ord_{\fp}(f) = \ord_{\fm}(f)$. By semi-continuity of multiplicity we can assume $\fp=\fm$ is a maximal ideal of $\Spec(R)$. By Zariski's Nullstellensatz, every maximal ideal of $R$ has the form $(x_1-a_1,\ldots,x_n-a_n)$ for some $(a_1,\ldots,a_n)\in V(P)\subseteq \mathbb{A}^n_k$. After a linear change of coordinates, a process that does not change the homogeneous coordinate ring $S=k[X_0,X_1,\ldots,X_n]/\prescript{h}{}{P}$,\footnote{See \ref{subsection homogenization} for more on homogenization.} we may assume $\fm=(x_1,\ldots,x_n)$.

   Suppose that $f\in \fm$ and $\ord_\fm(f)=t$. There exists a lift of $f$ in the polynomial ring $k[x_1,\ldots,x_n]=k\left[\frac{X_1}{X_0},\ldots,\frac{X_n}{X_0}\right]$, $f=f_t+f_{t+1}+\cdots + f_{t+c}$, so that each element $f_i$, if a non-zero element in $R$, is a homogeneous polynomial of degree $i$ with respect to the maximal ideal $(x_1,\ldots, x_n)\subseteq k[x_1,\ldots,x_n]$. For each $t\leq i \leq t+c$ let $F_i$ be the corresponding homogeneous polynomial of $k[X_1,\ldots,X_n]\subseteq k[X_0,X_1,\ldots,X_n]$ and $F=F_t+\cdots +F_{t+c}$.

   The ring $S$ is a standard graded normal domain. By Proposition~\ref{proposition Rees valuation of standard graded normal domain}, the (non-homogeneous) local ring $S_{\MM}$ enjoys the property that the collection of Rees valuations of the maximal ideal $\MM S_{\MM}$ is a $1$-element set, $\RR_{\MM S_{\MM}}=\{\omega\}$, $\omega(g)=\max\{n\in \NN\mid g\in \MM^nS_{\MM}\}$, and $e(S_{\MM}/gS_{\MM}) = e(S)\omega(g)$ for all $0\not =g\in \MM S_{\MM}$. 
   
   If $0\leq i\leq t-1$ then $F\in \MM^i$ and represents the $0$-element of $\MM^i/\MM^{i+1}$. The element $F$ and $F_t$ represent the same element of $\MM^t/\MM^{t+1}$. Note that $\frac{F_t}{X_0^t}=f_t$ is a non-zero element of $R$. Therefore $F_t$ is a nonzero element of $S\cong \Gr_{\MM}(S)$. Hence the image of $F$ in $\MM^t/\MM^{t+1}\subseteq S$ is nonzero, $\ord_{\MM}(F)=t$, and
   \begin{equation}\label{multiplicity single valuation}
   e\left(\left(\frac{S}{FS}\right)_{\MM}\right) = e(S)t=e(S_{\MM})\ord_\fm(f).
   \end{equation}
   Note that $S/(X_1,\ldots,X_n)\cong k[X_0]$, in particular the ideal $(X_1,\ldots,X_n)$ is a prime ideal. By semi-continuity of multiplicity,
   \begin{equation}\label{multiplicity semicontinuity}
  e\left(\left(\frac{S}{FS}\right)_{(X_1,\cdots,X_n)}\right)\leq e\left(\left(\frac{S}{FS}\right)_{\MM}\right).
   \end{equation}

   The element $X_0$ avoids the prime ideal $(X_1,\ldots,X_n)S$. Therefore as subsets of the fraction field of $S$;
   \[
   S\left[\frac{X_1}{X_0},\ldots,\frac{X_n}{X_0}\right]_{\left(\frac{X_1}{X_0},\ldots,\frac{X_n}{X_0}\right)} = S_{(X_1,\ldots,X_n)}.
   \]
   By Proposition~\ref{proposition: homogeneization properties wrt localization}, $S\left[\frac{X_1}{X_0},\ldots,\frac{X_n}{X_0}\right]= R[X_0]$. Let 
   \[
   F'=\frac{F}{X_{0}^t} =f_t+X_0f_t+\cdots + X_0^cf_{t+c}.
   \]
   By the above,
   \begin{equation}
       \label{multiplicity under isom}
       e\left(\left(\frac{S}{FS}\right)_{(X_1,\cdots,X_n)}\right) = e\left(\left(\frac{R[X_0]}{F'R[X_0]}\right)_{(x_1,\ldots,x_n)R[X_0]}\right).
   \end{equation}
   \begin{claim}
       \label{claim 1-x avoids associated primes}
       Let $A=\frac{R[X_0]}{F'R[X_0]}$. Then 
       \[
       e\left(\left(A\right)_{\fm A}\right)= e\left(\left(\frac{A}{(X_0-1)A}\right)_{(\fm,X_0-1)A}\right) = e\left(\left(\frac{R}{fR}\right)_{\fm}\right).
       \]
   \end{claim}

   \begin{proof}[Proof of Claim]
   Observe that
   \[
   \frac{A}{(X_0-1)A} \cong \frac{R}{(f_t+f_{t+1}+\cdots + f_{t+c})R}
   \]
   and consequently
   \begin{equation}\label{multiplicity final equality}
   e\left(\left(\frac{A}{(X_0-1)A}\right)_{(x_1,\ldots,x_n,X_0-1)}\right) = e\left(\left(\frac{R}{(f_t+f_{t+1}+\cdots + f_{t+c})R}\right)_{\fm}\right) = e\left(\left(\frac{R}{fR}\right)_{\fm}\right).
   \end{equation}
    It remains to show $e\left(\left(A\right)_{\fm A}\right)= e\left(\left(\frac{A}{(X_0-1)A}\right)_{(\fm,X_0-1)A}\right)$.

   Recall that $F'=f_t+X_0f_{t+1}+\cdots + X_0^cf_{t+c}$. We change the expansion of $F'$ around $X_0=0$ to an expansion around $X_0=1$. Then if $g_t,\ldots,g_{t+c}$ are the polynomials described by Corollary~\ref{corollary extended rees algebra multiplicity computation}, then
   \[
   F' = (f_t+\cdots + f_{t+c}) + (X_0-1)g_{t+1}+\cdots+(X_0-1)^cg_{t+c}.
   \]
   Hence if $\nu''\in \RR_{\fm R[X_0]}$ is a Gaussian extension of a valuation $\nu\in \RR_{\fm}$ to $R[X_0]=R[X_0-1]$ then
   \[
   \nu''(F') = \min\{\nu(f_t+f_{t+1}+\cdots + f_{t+c}), \nu(g_t),\ldots,\nu(g_{t+c})\}.
   \]
   By Corollary~\ref{corollary extended rees algebra multiplicity computation},
   \[
   \nu''(F') =\min\{\nu(f_t+f_{t+1}+\cdots + f_{t+c}), \nu(g_t),\ldots,\nu(g_{t+c})\} = \nu(f_t+f_{t+1}+\cdots + f_{t+c}).
   \]
   
   Theorem~\ref{theorem Rees multiplicity formula} and Corollary~\ref{corollary extended rees algebra multiplicity computation} imply
   \begin{align*}
   e\left(A_{\fm A}\right) &= \sum_{\nu''\in\RR_{\fm R[X_0]_{\fm R[X_0]}}}\nu''(F') e\left(\frac{\overline{R[X_0]_{\fm R[X_0]}[\fm T,T^{-1}]}}{Q_{\nu''}}\right) \\
   &= \sum_{\nu\in \RR_{\fm}}\nu(f)e\left(\frac{\overline{R[\fm T,T^{-1}]}}{Q_\nu}\right) {\footnotesize (\mbox{by Lemma~\ref{lemma reduction to an infinite field} (\ref{basically same exceptional prime}}))}\\
   &= e(R_\fm/fR_\fm).
   \end{align*}
   This completes the proof of the claim.
\renewcommand{\qedsymbol}{$\blacksquare$}
\end{proof}

We continue with the proof of the theorem and let $A$ be as in Claim~\ref{claim 1-x avoids associated primes} so that
\[
e\left(\left(\frac{A}{(X_0-1)A}\right)_{(\fm,X_0-1)A}\right) = e\left(\left(\frac{R}{fR}\right)_{\fm}\right)
\]
and then by (\ref{multiplicity under isom}),
\[
e\left(\left(\frac{R}{fR}\right)_{\fm}\right) = e\left(\left(\frac{S}{FS}\right)_{(X_1,\cdots,X_n)}\right).
\]
By (\ref{multiplicity semicontinuity}),
\[
e\left(\left(\frac{R}{fR}\right)_{\fm}\right)\leq e\left(\left(\frac{S}{FS}\right)_{\mathcal{M}}\right).
\]
Finally, by (\ref{multiplicity single valuation}),
\[
e\left(\left(\frac{R}{fR}\right)_{\fm}\right) \leq e(S_{\mathcal{M}})\ord_{\fm}(f).
\]
\end{proof}

\begin{corollary}[Main Theorem~\ref{Main Theorem eft over an algebraically closed field}]
    \label{corollary comparison of symbolic powers}
    Let $k$ be an algebraically closed field and $R$ a normal domain essentially of finite type over $k$. Let $S$ be a coordinate ring of an arithmetically normal projective closure of $\Spec(R)$, $\MM$ the homogeneous maximal ideal of $S$, and $e(S)$ the multiplicity of $S$ with respect to $\MM$. Then for all prime ideals $\fp\subseteq \fq\in\Spec(R)$ and for all $n\geq 1$,
    \[
    \fp^{(2e(S)n)}\subseteq \fp^{(e(S)n+1)}\subseteq \fq^{(n)}.
    \]
\end{corollary}

\begin{proof}
    It is clear that $\fp^{(2e(S)n)}\subseteq \fp^{(e(S)n+1)}$. If $f\in \fp^{(e(S)n+1)}$ then by semi-continuity of multiplicity
    \[
    e(S)n+1\leq e\left(\frac{R_\fp}{fR_\fp}\right)\leq e\left(\frac{R_\fq}{fR_\fq}\right).
    \]
    By Theorem~\ref{theorem uniform izumi arithmetically normal}, $f\in\fq^{(n)}$.
\end{proof}

\begin{corollary}\label{corollary uniform izumi rees arithmetically normal}
    Let $k$ be an algebraically closed field and $R$ a normal domain essentially of finite type over $k$. Let $S$ be a coordinate ring of an arithmetically normal projective closure of $\Spec(R)$, $\MM$ the homogeneous maximal ideal of $S$, and $e(S)$ the multiplicity of $S$ with respect to $\MM$. Then for every $\fp\in\Spec(R)$, for every pair of distinct Rees valuations $\nu_1,\nu_2\in \RR_{\fp R_\fp}$ of the maximal ideal of $R_\fp$, and $0\not = f\in\fp$,
    \[
    \nu_1(f)\leq (e(S)-1)\nu_2(f).
    \]
\end{corollary}

\begin{proof}
    The corollary is a direct application of Theorem~\ref{theorem uniform izumi arithmetically normal} and Proposition~\ref{proposition of rees order ideal thm and izumi rees} (\ref{multiplicity implies  izumi}).
\end{proof}

\subsection{Reduction to an algebraically closed field} The valuation criteria of the Uniform Izumi-Rees Property (see Definition~\ref{definition uniform izumi rees}), Lemma~\ref{lemma separable field extension}, and Lemma~\ref{lemma purely insep field extension} are utilized to reduce the Uniform Izumi-Rees Property for an algebra essentially of finite type over a field to the Uniform Izumi-Rees Property to an algebra essentially of finite type over an algebraically closed field.

If $k$ is a field and $\overline{k}$ an algebraic closure of $k$, then the extension $k \to \overline{k}$ can be factored as 
\[
k \longrightarrow k_{\sep} \longrightarrow \overline{k},
\]
where $k_{\sep}$ denotes the separable closure of $k$, and the extension $k_{\sep} \to \overline{k}$ is purely inseparable. To establish the Uniform Izumi Property for a normal domain essentially of finite type over $k$, we examine the behavior of Rees valuations and the Uniform Izumi–Rees Property (Definition~\ref{definition uniform izumi rees}) separately along separable and inseparable field extensions. We remark that if $k\to k'$ is an algebraic extension, $R$ an algebra essentially of finite type over $k$, and $R' = R\otimes_kk'$, then $R\to R'$ is a faithfully flat ring extension. By lying over, the map of spectrum $\Spec(R')\to \Spec(R)$ is surjective, and for all $\fp\in \Spec(R)$, the minimal primes of the extended ideal $\fp R'$ have the same height as $\fp$, see \cite[Section~2.2]{SwansonHuneke} for necessary details.

Now let $k \to k'$ be a separable field extension. If $R$ is a normal domain essentially of finite type over $k$, then $R \otimes_k k'$ is not necessarily a domain, but instead decomposes as a product of normal domains (see \cite[Corollary~2.1.13 and Theorem~19.4.3]{SwansonHuneke}). However, the study of the Uniform Izumi–Rees Property naturally extends to products of normal domains: since the Uniform Izumi–Rees Property is local, a product of normal domains satisfies the Uniform Izumi-Rees Property if and only if each factor does.

\begin{remark}
    \label{remark about seperable field extensions}
    Let $k$ be a field, $k\to k'$ an algebraic separable field extension, $R$ an integral domain essentially of finite type over $k$ with field of fractions $K$, and $\overline{R}$ the normalization of $R$. Then $R\otimes_k k'\subseteq \overline{R}\otimes_k k'\subseteq K\otimes_k k'$ and $K\otimes_kk'$ is the total ring of fractions of $R\otimes_kk'$ and is a product of separable field extensions of $K$. It follows that $R\otimes_kk'\to \overline{R}\otimes_k k'$ is finite, birational, and $\overline{R}\otimes_k k'$ is normal by \cite[Theorem~19.4.3]{SwansonHuneke}. Therefore $\overline{R\otimes_k k'}\cong \overline{R}\otimes_k k'$. 
\end{remark}

\begin{lemma}
    \label{lemma separable field extension}
    Let $R$ be a normal domain essentially of finite type over a field $k$. Let $k\to k'$ be a separable algebraic extension of $k$ and $R'=R\otimes_k k'$. If $R'$ enjoys the Uniform Izumi-Rees Property with Uniform Izumi-Rees bound $E$, then $R$ enjoys the Uniform Izumi-Rees Property with Uniform Izumi-Rees bound $E$.
\end{lemma}

\begin{proof}
    The algebra $R'$ is integrally closed in its total ring of fractions, \cite[Theorem~19.4.3]{SwansonHuneke}, and therefore is a product of normal domains, \cite[Corollary~2.1.13]{SwansonHuneke}.

    Consider the map of extended Rees algebras 
        \[
        \varphi: R[\fp T, T^{-1}]\to R'[\fp' T, T^{-1}]
        \]
        and the induced map of normalizations
        \[
        \overline{\varphi}: \overline{R[\fp T, T^{-1}]}\to \overline{R'[\fp' T, T^{-1}]}.
        \]
        By Remark~\ref{remark about seperable field extensions}, 
        \[
        \overline{R'[\fp' T, T^{-1}]}\cong \overline{R[\fp T, T^{-1}]}\otimes_k k'.
        \]

    \begin{claim}\label{claim Rees vals along separable extension}
        Let $\fp\in\Spec(R)$ and $\fp'=\fp R'$. For each $\nu\in\mathcal{R}_\fp$ and $\omega\in \mathcal{R}_{\fp'}$ let $Q_\nu$ and $Q_\omega$ denote the corresponding height $1$ prime ideals in $\overline{R[\fp T,T^{-1}]}$ and $\overline{R'[\fp' T, T^{-1}]}$ respectively.
        \begin{enumerate}
            \item There is a partition of the Rees valuations $\fp'$, indexed by the Rees valuations of $\fp$,
            \[
            \mathcal{R}_{\fp'} = \bigcup_{\nu \in \mathcal{R}_\fp}\Lambda_\nu
            \]
            with the defining property that if $\nu'\in \mathcal{R}_{\fp'}$, then $\nu' \in \Lambda_\nu$ if and only if $\nu'(Q_\nu)\geq 1$ if and only if $\nu'(Q_\nu)= 1$.
            \item\label{equality of Rees vals along separable} If $\nu\in \mathcal{R}_\fp$, $\nu'\in\Lambda_{\nu}$, and $f\in R$, then    
            \[
            \nu(f) = \nu'(f).
            \]
            \item\label{choose Rees vals over common prime lying over} Let $\nu_1,\nu_2\in \RR_{\fp R_{\fp}}$ and $\fq$ a minimal prime of $\fp'$. There exists $\omega_1,\omega_2 \in \RR_{\fq R_{\fq}}$ so that $Q_{\omega_1}\cap \overline{R[\fp T,T^{-1}]} = Q_{\nu_1}$ and $Q_{\omega_2}\cap \overline{R[\fp T,T^{-1}]}  =Q_{\nu_2}$.
        \end{enumerate}
    \end{claim}

\begin{proof}[Proof of Claim]
    In general, if $S$ is a $k$-algebra, $S'=S\otimes_kk'$, and $\fp\subseteq S$ a prime ideal, then $S/\fp\to S'/\fp S'$ is flat and algebraic, implying that $S'/\fp S'$ injects into $\overline{S/\fp}\otimes_k k'$, the latter of which is a product of normal domains by \cite[Theorem~19.4.3]{SwansonHuneke} and \cite[Corollary~2.1.13]{SwansonHuneke}. In particular, $\fp S'$ is a reduced ideal whose components have a common height. In the context of extended Rees algebras and Rees valuations, for each Rees valuation $\nu\in \mathcal{R}_\fp$, there are unique prime components of $Q_{\nu'_1},\ldots, Q_{\nu'_t}$ of $T^{-1}\overline{R'[\fp' T, T^{-1}]}$ so that
    \[
    Q_\nu \overline{R'[\fp' T, T^{-1}]} = Q_{\nu_1'}\cap \cdots \cap Q_{\nu_t'}.
    \]
    Hence there is a partition of the Rees valuations $\fp'$, indexed by the Rees valuations of $\fp$,
    \[
    \mathcal{R}_{\fp'} = \bigcup_{\nu \in \mathcal{R}_\fp}\Lambda_\nu
    \]
    with the defining property that if $\nu'\in \mathcal{R}_{\fp'}$, then $\nu' \in \Lambda_\nu$ if and only if $\nu'(Q_\nu)\geq 1$ if and only if $\nu'(Q_\nu)= 1$ as the extension of $Q_{\nu}$ to $\overline{R'}[\fp'T,T^{-1}]$ is reduced. This completes part (a) of the claim.

    Now suppose that $\nu \in \mathcal{R}_\fp$ is a Rees valuation of $\fp$. Let $\Lambda_\nu=\{\omega_1,\ldots,\omega_t\}$ be the corresponding Rees valuations of $\fp'$, let $W$ be the complement of the union of the prime ideals $Q_{\omega_i}$ in $\overline{R'[\fp'T, T^{-1}]}$ and consider the map of localizations
    \[
    \overline{\varphi}_{Q_{\nu}}: \overline{R[\fp T,T^{-1}]}_{Q_\nu}\to \overline{R'[\fp'T, T^{-1}]}_{W}.
    \]
    Let $f\in R$ and consider the principal ideal
    \[
    f\overline{R[\fp T,T^{-1}]}_{Q_\nu} = \left(Q_\nu^{\nu(f)}\right)_{Q_\nu},
    \]
    and its expansion under $\overline{\varphi}_{Q_\nu}$,
    \[
    f\overline{R'[\fp' T,T^{-1}]}_{W} = \left(\bigcap_{i=1}^tQ_{\omega_i}^{\omega_i(f)}\right)_{W}.
    \]
    Therefore $\omega_i(f) = \nu(f)\omega_i(Q_\nu) = \nu(f)$  for each $1\leq i \leq t$, as claimed.

    Only part (\ref{choose Rees vals over common prime lying over}) of the claim remains to be proven. Assume that 
    \[
    \fp' = \fp R' = \fq_1 \cap \cdots \cap \fq_t
    \]
    is the minimal decomposition of $\fp'$ as a reduced ideal of $R'$. Since $R \to R'$ is algebraic, the primes $\{\fq_1, \fq_2, \ldots, \fq_t\}$ are precisely the primes of $\Spec(R')$ lying over $\fp \in \Spec(R)$. The claim in (\ref{choose Rees vals over common prime lying over}) is that for each $1 \leq i \leq t$ there exist Rees valuations $\omega_1, \omega_2 \in \mathcal{R}_{\fq_i}$ such that 
    \[
    Q_{\omega_1} \cap \overline{R[\fp T, T^{-1}]} = Q_{\nu_1}
    \quad \text{and} \quad
    Q_{\omega_2} \cap \overline{R[\fp T, T^{-1}]} = Q_{\nu_2}.
    \]

    The minimal prime components of $T^{-1}\overline{R[\fp T, T^{-1}]}$ that contract to $\fp$ in $R$ are in bijection with the minimal prime components of 
    \[
    \frac{\overline{R_{\fp}[\fp R_{\fp} T, T^{-1}]}}{T^{-1}\overline{R_{\fp}[\fp R_{\fp} T, T^{-1}]}}.
    \]
    Similarly, the minimal prime components of $T^{-1}\overline{R'[\fp R'T, T^{-1}]}$ lying over a minimal prime of $\fp' = \fp R'$ correspond to the minimal prime components of 
    \[
    \frac{\overline{R'_{\fp}[\fp R'_{\fp} T, T^{-1}]}}{T^{-1}\overline{R'_{\fp}[\fp R'_{\fp} T, T^{-1}]}}.
    \]
    
    Now, since $\fp' = \fp R' = \fq_1 \cap \cdots \cap \fq_t$ is the minimal decomposition of $\fp'$ as a reduced ideal of $R'$, and $R \to R'$ is algebraic, we have
    \[
    \frac{R'_{\fp}}{\fp R'_{\fp}} \cong \bigtimes_{i=1}^t \frac{R'_{\fq_i}}{\fq_i R'_{\fq_i}}.
    \]
    
    By Remark~\ref{remark about seperable field extensions}, it follows that
    \[
    \frac{\overline{R_{\fp}[\fp R_{\fp} T, T^{-1}]}}{T^{-1}\overline{R_{\fp}[\fp R_{\fp} T, T^{-1}]}} \otimes_k k'
    \;\;\cong\;\;
    \frac{\overline{R'_{\fp}[\fp R'_{\fp} T, T^{-1}]}}{T^{-1}\overline{R'_{\fp}[\fp R'_{\fp} T, T^{-1}]}}
    \;\;\cong\;\;
    \bigtimes_{i=1}^t \frac{\overline{R'_{\fq_i}[\fq_i R'_{\fq_i} T, T^{-1}]}}{T^{-1}\overline{R'_{\fq_i}[\fq_i R'_{\fq_i} T, T^{-1}]}}.
    \]
    
    Since $k \to k'$ is separable and algebraic, we deduce that 
    \[
    \frac{\overline{R_{\fp}[\fp R_{\fp} T, T^{-1}]}}{T^{-1}\overline{R_{\fp}[\fp R_{\fp} T, T^{-1}]}}
    \;\;\longrightarrow\;\;
    \bigtimes_{i=1}^t 
    \frac{\overline{R'_{\fq_i}[\fq_i R'_{\fq_i} T, T^{-1}]}}{T^{-1}\overline{R'_{\fq_i}[\fq_i R'_{\fq_i} T, T^{-1}]}}
    \]
    is algebraic. Equivalently, for each $1 \leq i \leq t$, 
    \[
    \frac{\overline{R_{\fp}[\fp R_{\fp} T, T^{-1}]}}{T^{-1}\overline{R_{\fp}[\fp R_{\fp} T, T^{-1}]}}
    \;\;\longrightarrow\;\;
    \frac{\overline{R'_{\fq_i}[\fq_i R'_{\fq_i} T, T^{-1}]}}{T^{-1}\overline{R'_{\fq_i}[\fq_i R'_{\fq_i} T, T^{-1}]}}
    \]
    is algebraic. 
    
    Therefore, the minimal primes $Q_{\nu_1}$ and $Q_{\nu_2}$ of $T^{-1}\overline{R_{\fp}[\fp R_{\fp} T, T^{-1}]}$ are contractions of minimal primes of $T^{-1}\overline{R'_{\fq_i}[\fq_i R'_{\fq_i} T, T^{-1}]}$. These, in turn, correspond to minimal primes of $T^{-1}\overline{R'[\fq_i R'T, T^{-1}]}$ that contract to $\fq_i$ in $R'$, completing the proof of the claim.
    \renewcommand{\qedsymbol}{$\blacksquare$}
    \end{proof}
    
    We continue with the proof of the lemma. Let $\nu_1,\nu_2\in \mathcal{R}_\fp$ be Rees valuations of $\fp$ centered on $\fp$ and choose Rees Valuations $\nu_1',\nu_2'$ of $\fp'$ belonging to $\Lambda_{\nu_1}$ and $\Lambda_{\nu_2}$ respectively. By (\ref{choose Rees vals over common prime lying over}) of Claim~\ref{claim Rees vals along separable extension}, we can choose $\nu_1'$ and $\nu_2'$ to be centered on a common minimal prime component of $\fp'$. If $E$ is a Uniform Izumi-Rees bound of $R'$, then by (\ref{equality of Rees vals along separable}) of Claim~\ref{claim Rees vals along separable extension},
    \[
    \nu_1(f) = \nu_1'(f) \leq E\nu_2'(f) = E\nu_2(f).
    \]
    Therefore the Uniform Izumi-Rees Property of $R'$ descends to $R$ with Uniform Izumi-Rees Bound $E$.
\end{proof}

Let $k$ be a field of prime characteristic $p>0$ and $k\to k'$ an algebraic and purely inseparable field extension of $k$, i.e., for each $\alpha\in k'$ there exists $e\in\NN$ so that $\alpha^{p^e}\in k$. Let $R$ be an algebra essentially of finite type over $k$, $R' = R\otimes_k k'$. If $\fp\in \Spec(R)$, then the extended ideal $\fp R'$ is easily checked to be primary to $\sqrt{\fp R'}$. Therefore $\fp R'$ omits a unique minimal prime, namely $\sqrt{\fp R'}$, and the induced map of spectrum $\Spec(R')\to \Spec(R)$ is a bijection.

Unlike a separable base change, if $R$ is a normal domain essentially of finite type over $k$, and $k\to k'$ purely inseparable, the algebra $R\otimes_k k'$ no longer has to be integrally closed in its total ring of fractions as $R\otimes_k k'$ can be non-reduced. The following lemma reduces the study of the Uniform Izumi-Rees Property of $R$ to a normal domain obtained from $R\otimes_k k'$, namely the normalization of the integral domain $(R\otimes_k k')/\sqrt{0}$.

\begin{lemma}
    \label{lemma purely insep field extension}
    Let $R$ be a normal domain essentially of finite type over a field $k$ of prime characteristic $p>0$. Let $k\to k'$ be a purely inseparable algebraic extension of $k$, $R'=(R\otimes_k k')/\sqrt{0}$, and $\overline{R'}$ the normalization of $R'$. If $I\subseteq R$ is an ideal and $I'=IR'$ then there is a bijection of Rees valuations $\psi:\RR_I\to \RR_{I'}$. Moreover, for all $\nu\in\mathcal{R}_I$, if $\nu'=\psi(\nu)$, $\fp_\nu$ the center of $\nu$ in $R$, then $\sqrt{\fp_\nu R'}=\fp_{\nu'}$ is the center of $\nu'$ in $R'$, and for all $f\in R$,
    \[
    \frac{\nu(f)}{\nu(I)}=\frac{\nu'(f)}{\nu'(I')}.
    \]
    Let $e$ and $e'$ be respective uniform upper bounds of the Hilbert-Samuel multiplicities of the localizations of $R$ and $\overline{R'}$ at their prime ideals. If $\overline{R'}$ has the Uniform Izumi-Rees Property with Uniform Izumi-Rees Bound $E$, then $R$ has the Uniform Izumi-Rees Property with Uniform Izumi-Rees Bound $Eee'$.
\end{lemma}

\begin{proof}
   Let $K'$ be the field of fractions of $R'$. Then $R\to R'$  and $K\to K'$ are purely inseparable. The extension $R \to \overline{R'}$ is purely inseparable. Indeed, if $f\in K'$ and satisfies a polynomial equation
   \[
   f^{t}+a_1f^{t-1}+\cdots+a_{t-1}f+a_t=0
   \]
   with each coefficient $a_i$ belonging to $R'$, then we can choose $e\gg 0$ so that $f^{p^e}\in K$ and $a_i^{p^e}\in R$. If we raise the above equation of integral dependence to $p^e$, then
   \[
   f^{p^et}+a_1^{p^e}f^{p^e(t-1)}+\cdots+a_{t-1}^{p^e}f^{p^e}+a_t^{p^e}=0.
   \]
   Therefore $f^{p^e}$ belongs to the normalization of $R$. The ring $R$ is assumed to be normal, therefore $f^{p^e}\in R$. \footnote{If $R$ was not assumed to be normal, then the argument shows that the map of normalizations $\overline{R}\to \overline{R'}$ is purely inseparable.}

    Consider the map of extended Rees algebras 
    \[
    \varphi: R[I T, T^{-1}]\to R'[I' T, T^{-1}]
    \]
    and the induced map of normalizations
    \[
    \overline{\varphi}: \overline{R[I T, T^{-1}]}\to \overline{R'[I' T, T^{-1}]}.
    \]
    The extension $\varphi$ is purely inseparable. Therefore $\overline{\varphi}$ is a purely inseparable extension by the footnote. Hence there is a bijection of the components of $T^{-1}\overline{R[I T, T^{-1}]}$ with the components of $T^{-1}\overline{R'[I' T, T^{-1}]}$. Equivalently, there is a bijection of Rees valuations
    \[
    \psi: \mathcal{R}_I\to \mathcal{R}_{I'}
    \]
    defined as follows: If $Q_\nu$ is the component of $T^{-1}\overline{R[I T, T^{-1}]}$ corresponding to $\nu$ and $Q_{\nu'}$ is the prime ideal $\sqrt{Q_\nu \overline{R'[I' T, T^{-1}]}}$, then $\psi(\nu)=\nu'$.

    Let $W$ be the complement of the union of the components of $T^{-1} \overline{R[I T, T^{-1}]}$. Consider the map $\overline{\varphi}_W$ of localizations
    \[
    \overline{\varphi}_W:\overline{R[IT, T^{-1}]}_W\to \overline{R'[I' T, T^{-1}]}_{W}.
    \]
    We examine the decomposition of the localized principal ideal
    \[
    T^{-1}\overline{R[I T, T^{-1}]}_W = \left(\bigcap_{\nu\in\mathcal{R}_I} Q_\nu^{\nu(I)}\right)_W
    \]
    and its decomposition under $\overline{\varphi}_W$,
    \[
    T^{-1}\overline{R'[I' T, T^{-1}]} = \left(\bigcap_{\nu\in\mathcal{R}_\fp} Q_{\nu'}^{\nu'(I')}\right)_{W}.
    \]
    Then the bijection of components under expansion implies
    \begin{align}
    \label{equation of Rees val under purely insep}
    \nu'(I') = \nu(I)\nu'(Q_\nu).
    \end{align}
    Let $f\in R$ and consider the principal ideal
    \[
    f\overline{R[I T, T^{-1}]}_W = \left(\bigcap_{\nu\in\mathcal{R}_I} Q_\nu^{\nu(f)}\right)_W,
    \]
    and its image under $\overline{\varphi}_W$ in $\overline{R'[I' T, T^{-1}]}_{W}$,
    \[
    f\overline{R'[I' T, T^{-1}]}_{W} = \left(\bigcap_{I'\in\mathcal{R}_{I'}} Q_{\nu'}^{\nu'(f)}\right)_{W}.
    \]
    The bijection of exceptional components and (\ref{equation of Rees val under purely insep}) implies
    \[
    \nu'(f)=\nu(f)\nu'(Q_\nu)=\nu(f)\frac{\nu'(I)}{\nu(I)}.
    \]
    Therefore
    \[
    \frac{\nu(f)}{\nu(I)} = \frac{\nu'(f)}{\nu'(I')}
    \]
    as claimed.

    Suppose that $I=\fp\in\Spec(R)$. Let $\nu_1,\nu_2\in \mathcal{R}_\fp$ be Rees valuations of $\fp$, both centered on $\fp$, and $\nu_1',\nu_2'\in \mathcal{R}_{\fp'}$ the corresponding Rees valuations of $\fp'$, both of which are necessarily centered on $\fp'$. By the above and the assumption that $R'$ enjoys the Uniform Izumi-Rees Property with Uniform Izumi-Rees Bound $E$,
    \[
    \frac{\nu_1(f)}{\nu_1(\fp)} = \frac{\nu'_1(f)}{\nu'_1(\fp')}\leq \frac{E\nu_2'(f)}{\nu_1'(\fp')}=\frac{E\nu_2'(\fp')}{\nu'_1(\fp')}\cdot \frac{\nu_2'(f)}{\nu_2'(\fp')} = \frac{E\nu_2'(\fp')}{\nu_1'(\fp)}\cdot \frac{\nu_2(f)}{\nu_2(\fp)}\leq E\nu_2'(\fp')\nu_2(f).
    \]
    Therefore
    \[
    \nu_1(f)\leq E\nu_1(\fp)\nu_2'(\fp')\nu_2(f).
    \]
    The values $\nu_1(\fp)$ and $\nu_2'(\fp')$ are bounded from above by the Hilbert-Samuel multiplicities of $R_\fp$ and $R'_{\fp'}$ respectively by Corollary~\ref{corollary Rees multiplicity formula}, values that are bounded from above by $e$ and $e'$ respectively. Therefore $R$ enjoys the Uniform Izumi-Rees Property with Uniform Izumi-Rees Bound $Eee'$.
\end{proof}

\begin{theorem}[Main Theorem~\ref{Main Theorem Uniform Izumi Rees}]
    \label{theorem uniform izumi rees eft over a field}
    Let $k$ be a field and $R$ be a normal domain essentially of finite type over $k$. Then $R$ enjoys the Uniform Izumi-Rees Property.
\end{theorem}

\begin{proof}
     If $\overline{k}$ is an algebraic closure of $k$, then we can factor $k\to \overline{k}$ by a separable field extension and then a purely inseparable extension. By Lemma~\ref{lemma separable field extension} and Lemma~\ref{lemma purely insep field extension}, we can replace $k$ with $\overline{k}$ and $R$ with the normalization of $(R\otimes_k\overline{k})/\sqrt{0}$ so that our ring $R$ is essentially of finite type over an and algebraically closed field. The theorem then follows by Theorem~\ref{theorem uniform izumi arithmetically normal} and Corollary~\ref{corollary multiple definitions of Uniform Izumi-Rees Property}.
\end{proof}

\section{Zariski-Nagata for Singularities}\label{Section Uniform Chevalley}
 This section contains the proof of Main Theorem~\ref{Main Theorem Improved Uniform Chevalley general}. Let $R$ be a Noetherian ring, $I\subseteq R$, $n\in\NN$, and $W$ the complement of the union of the associated primes of $I$. The \emph{$n$th symbolic power of $I$} is the ideal $I^{(n)}:=I^nR_W\cap R$. We begin with a lemma.

\begin{lemma}
\label{Lemma fixed power of m contains all fixed symbolic powers of all ideals improved}
Let $(R,\fm,k)$ be an excellent normal local domain and $E$ an Izumi-Rees bound of the Rees valuations $\RR_\fm$ of the maximal ideal $\fm$. If $t\in \mathbb{N}$ then for every ideal $I\subseteq R$, if $W$ is the complement of the union of the minimal primes of $I$,
\[
\overline{I^{Ete(R)^2}}R_W\cap R \subseteq \overline{\fm^t}.
\]
\end{lemma}

\begin{proof}
By Lemma~\ref{lemma reduction to an infinite field} and Rees' Valuation criterion for containment in integral closure, see Theorem~\ref{theorem standard facts about integral closure and valuations} (\ref{Valuation Criteria}), we may assume $R$ has infinite residue fields. If $R$ is at most $1$-dimensional, then $R$ is either a field or a discrete valuation ring. Either case, the content of the lemma is trivial. In what follows, $R$ has dimension at least $2$.

Suppose $\fp$ is a minimal prime of $I$. Then 
\[
\overline{I^n}R_W\cap R\subseteq \overline{\fp^n}R_{W}\cap R\subseteq \overline{\fp^n}R_{\fp}\cap R.
\]
We therefore may reduce our considerations to $I=\fp\in\Spec(R)$ and show
\[
\overline{\fp^{Ete(R)^2}}R_\fp\cap R \subseteq \overline{\fm^t}.
\]

\begin{claim}\label{Izumi ideal containment} For each Rees valuation $\nu\in\RR_{\fm}$ of the maximal ideal $\fm$, there is a containment of ideals
\[
    I_{\nu\geq Ete(R)}\subseteq \overline{\fm^t}.
\]
\end{claim}

\begin{proof}[Proof of Claim]
    If $f\in I_{\nu\geq Ete(R)}$ then $\nu(f)\geq Ete(R)$. For all $\omega\in\RR_{\fm R}$, $\nu(f)\leq E \omega(f)$, hence
    \[
    E\omega(f)\geq \nu(f)\geq Ete(R).
    \]
    In particular, $\omega(f)\geq te(R)$.  By Corollary~\ref{corollary Rees multiplicity formula}, $e(R)\geq \omega(\fm)$. 
    Therefore, for every $\omega\in\RR_{\fm}$,
    \begin{equation}
    \label{equation large enough rees vals for q}
    \omega(f)\geq t\omega(\fm).
    \end{equation}    
    Thus $f\in\overline{\fm^t}$ by Rees' valuation criteria for containment in $\overline{\fm^t}$, see Theorem~\ref{theorem standard facts about integral closure and valuations} part (\ref{Valuation Criteria}). This completes the proof of the claim.
\renewcommand{\qedsymbol}{$\blacksquare$}
\end{proof}

Continue the proof of the lemma and let $f \in \overline{\mathfrak{p}^{Ete(R)^2}}R_\fp \cap R$ and suppose by way of contradiction that $f\not\in\overline{\fm^t}$. Because $f\in \overline{\mathfrak{p}^{Ete(R)^2}}R_\fp \cap R$, by Rees' Order Ideal Theorem, Theorem~\ref{theorem Rees multiplicity formula}, there are constants $d_{\nu}\in \NN$ associated to each Rees valuation $\nu\in \RR_{\fp R_{\fp}}$ so that
\[
e(R_{\fp}/fR_{\fp}) = \sum_{\nu\in \RR_{\fp R_{\fp}}}\nu(f)d_{\nu}\geq \nu(\overline{\fp^{Ete(R)^2}}R_{\fp})d_{\nu} = Ete(R)^2\sum_{\nu\in \RR_{\fp R_{\fp}}}\nu(\fp R_{\fp})d_{\nu}.
\]
By Corollary~\ref{corollary Rees multiplicity formula}, $\sum_{\nu\in \RR_{\fp R_{\fp}}}\nu(\fp R_{\fp})d_{\nu} = e(R_{\fp})$. Therefore
\[
e(R_{\fp}/f R_{\fp})\geq Ete(R)^2e(R_{\fp})\geq Ete(R)^2.
\]
All associated primes of $R/fR$ are height $1$ primes of the catenary ring $R$. Hence, by the semi-continuity of multiplicity and by Rees' Order Ideal Theorem, Theorem~\ref{theorem Rees multiplicity formula}, for each $\nu\in\RR_{\fm}$ there exist a constant $d_\nu$, not depending on $f$, such that
\[
Ete(R)^2 \leq e(R_\mathfrak{p}/fR_\mathfrak{p}) \leq e(R/fR) = \sum_{\nu \in \mathcal{R}_{\fm}} \nu(f)d_\nu.
\]

We are assuming $f\not \in \overline{\fm^t}$. Hence $\nu(f)<Ete(R)$ for each Rees valuation $\nu\in\mathcal{R}_{\fm}$ by Claim~\ref{Izumi ideal containment}. Therefore
\[
Ete(R)^2< Et\sum_{\nu\in\mathcal{R}_{\fm}}e(R)d_\nu.
\]
The constants $d_\nu$ are such that $e(R) = \sum_{\nu\in\mathcal{R}_{\fm}}\nu(\fm) d_\nu$, see Corollary~\ref{corollary Rees multiplicity formula}. Therefore
\[
Ete(R)^2< Ete(R)\sum_{\nu\in\mathcal{R}_{\fm}}d_\nu\leq Ete(R)\sum_{\nu\in\mathcal{R}_{\fm}}\nu(\fm)d_\nu=Ete(R)^2,
\]
a contradiction.
\end{proof}

\begin{definition}
    Let $R$ be a Noetherian ring $I\subseteq R$ an ideal and $\fq$ a prime ideal containing $I$. Then the \emph{normalized order of $I$ with respect to $\fq$} is $\overline{\ord}_\fq(I):=\max\{t\in\NN\mid IR_\fq \subseteq \overline{\fq^t}R_\fq\}$.
\end{definition}

\begin{theorem}[Improved Uniform Chevalley Theorem Criteria]
    \label{theorem how to acheive uniform chevalley}
    Let $R$ be an excellent normal domain of finite Krull dimension that enjoys the following properties:
    \begin{enumerate}
        \item\label{UIR assumption} The Uniform Izumi-Rees Property with Uniform Izumi-Rees Bound $E$;
        \item\label{UBS assumption} The Uniform Brian\c{c}on-Skoda Property with Uniform Brian\c{c}on-Skoda Bound $B$;
        \item\label{TE assumption} There exists an element $0\not= c\in R$ and a constant $C$ so that for all ideals $J\subseteq R$, for all $n\in\NN$, if $W$ is the complement of the associated primes of $J$, then $c^n(\overline{J^{Cn}}R_W\cap R)\subseteq \overline{J^n}$;
        \item\label{UAR assumption} The containment $(c)\subseteq R$ enjoys the Uniform Artin-Rees Property with Uniform Artin-Rees Bound $A$.
    \end{enumerate}
    Let $e=\max\{e(R_\fp)\mid \fp\in\Spec(R)\}$. For every ideal $I\subseteq R$, if $W$ is the complement of the union of the associated primes of $I$, then for all $n\in\NN$, for all primes $\fq$ containing $IR_W\cap R$,
    \begin{itemize}
        \item $\overline{I^{CE(A+1)^2e^2n}}R_W\cap R\subseteq \overline{\fq^{\overline{\ord}_\fq(I)n}}R_\fq \cap R.$
        \item $I^{(CE(A+1)^2e^2(B+1)n)}\subseteq \fq^{(\overline{\ord}_\fq(I)(B+1)n-B)} \subseteq \fq^{(\overline{\ord}_\fq(I)n)}.$
    \end{itemize}
\end{theorem}

\begin{proof}
    First consider the case that $\overline{\ord}_\fq\left(\overline{I}\right)\geq A+1$. If $f\in \overline{I^{Cn}}R_W\cap R$ then 
    \begin{equation}\label{multiplied equation}
        c^nf\in \overline{I^{n}} \subseteq \overline{\fq^{\overline{\ord}_\fq(I)n}}R_\fq\cap R.
    \end{equation}   
    The constant $A$ is a Uniform Artin-Rees Bound of $(c)\subseteq R$. Therefore
    \[
     c(\fq^{\overline{\ord}_\fq(I)n}:_R c)= (c)\cap \fq^{\overline{\ord}_\fq(I)n} \subseteq c\fq^{\overline{\ord}_\fq(I)n-A}.
    \]
    Hence $(\fq^{\overline{\ord}_\fq(I)n}:_R c)\subseteq \fq^{\overline{\ord}_\fq(I)n-A}$. By induction, if $k\in\NN$ then $(\fq^{\overline{\ord}_\fq(I)n}:_R c^k)\subseteq \fq^{\overline{\ord}_\fq(I)n-Ak}$, so that when $n=k$, one has
    \[
    (\fq^{\overline{\ord}_\fq(I)n}:_R c^n)\subseteq \fq^{(\overline{\ord}_\fq(I)-A)n}.
    \]
    The containment persists upon taking integral closure and localization. Therefore,
    \begin{align}\label{containment upon integral closure and localization}
    (\overline{\fq^{\overline{\ord}_\fq(I)n}}R_\fq:_{R_\fq} c^n)\cap R\subseteq \overline{\fq^{(\overline{\ord}_\fq(I)-A)n}}R_\fq \cap R
    \end{align}
    for every $n\in\NN$. By (\ref{containment upon integral closure and localization}) and (\ref{multiplied equation}), for every $n\in\NN$, there is a containment of ideals
    \[
    \overline{I^{Cn}}R_W\cap R\subseteq \overline{\fq^{(\overline{\ord}_\fq(I)-A)n}}R_\fq \cap R.
    \]
    Apply the above containment with respect to $(A+1)n$,
    \begin{equation}\label{equation chevalley containment minus calculus exercise}
        \overline{I^{C(A+1)n}}R_W\cap R\subseteq \overline{\fq^{(\overline{\ord}_\fq(I)-A)(A+1)n}}R_\fq \cap R.
    \end{equation}

    The current assumption is that $\overline{\ord}_\fq\left(\overline{I}\right)\geq A+1$. An elementary inequality shows $(\overline{\ord}_\fq(I)-A)(A+1)n\geq \overline{\ord}_\fq(I)n$ for all $n\in\NN$.\footnote{If $x,a>0$ then $x\geq a+1$ if and only if $xa\geq (a+1)a$ if and only if $(x-a)(a+1)=xa-(a+1)a+x\geq x$.} Hence, if $n\in\NN$ then by (\ref{equation chevalley containment minus calculus exercise}) there are containments
    \begin{equation}\label{equation chevalley containment plus calculus exercise}
    \overline{I^{C(A+1)n}}R_W\cap R \subseteq \overline{\fq^{\overline{\ord}_\fq(I)n}}R_\fq \cap R.
    \end{equation}
    This completes the proof of the claim if $\overline{\ord}_\fq\left(\overline{I}\right)\geq A+1$.

    Now consider the case that $\overline{\ord}_\fq\left(\overline{I}\right)\leq A$. By Lemma~\ref{Lemma fixed power of m contains all fixed symbolic powers of all ideals improved} applied with respect to the constant $t=A+1$ and the maximal ideal $\fq R_\fq$ of the local ring $R_\fq$, 
    \begin{equation*}\label{equation direct application of lemma fixed power}
        \overline{I^{E(A+1)e^2}}R_W\cap R \subseteq \overline{I^{E(A+1)e(R_\fq)^2}}R_W\cap R \subseteq \overline{\fq^{A+1}}R_\fq\cap R.
    \end{equation*}    
    Therefore $\overline{\ord}_\fq\left(\overline{I^{E(A+1)e^2}}R_W\cap R\right)\geq A+1$ and we can apply the containment of (\ref{equation chevalley containment plus calculus exercise}) with respect to $J=\overline{I^{E(A+1)e^2}}R_W\cap R$. Then for every $n\in\NN$, there is a containment of ideals
    \[
    \overline{I^{CE(A+1)^2e^2n}}R_W\cap R = \overline{J^{C(A+1)n}}R_W\cap R\subseteq \overline{\fq^{\overline{\ord}_\fq(J)n}}R_\fq \cap R\subseteq \overline{\fq^{\overline{\ord}_\fq(I)n}}R_\fq \cap R.
    \]
    The remaining uniform ideal containment is an application of the assumption $R$ enjoys the Uniform Brian\c{c}on-Skoda Property with Uniform Brian\c{c}on-Skoda Bound $B$:
    \begin{align*}
         I^{(CE(A+1)^2e^2(B+1)n)}\subseteq \overline{I^{CE(A+1)^2e^2(B+1)n}}R_W\cap R&\subseteq \overline{\fq^{\overline{\ord}_\fq(I)(B+1)n}}R_\fq\cap R \\
         & \subseteq \fq^{(\overline{\ord}_\fq(I)(B+1)n-B)}\\
         &\subseteq \fq^{(\overline{\ord}_\fq(I)n)}.
    \end{align*}
\end{proof}

The above criteria for a ring to enjoy the improvement of the Uniform Chevalley Theorem that accounts for initial degree of vanishing required the existence of an element $0\not= c\in R$ and a constant $C$ so that for all ideals $I\subseteq R$, for all $n\in\NN$, if $W$ denotes the complement of union of the associated primes of $I$, then $c^n(\overline{I^{Cn}}R_W\cap R)\subseteq \overline{I^n}$. Similar notions have been studied by others in \cite{HHComparison, HKVfinite, HKV, HKAbelian, HunekeKatz}. In particular, it is known by experts that if $R$ is a domain that is either essentially of finite type over a field of characteristic $0$ or is of prime characteristic $p>0$ and $F$-finite, then any $0\not=c\in R$ with the property that $R_c$ is non-singular will have a power with the desired property. We sketch a proof and provide suitable references for details. Properties of gamma constructions given in \cite{MurayamaGamma} then allow us to extend the result to rings essentially of finite type over any field, i.e., we do not require a restriction to $F$-finite rings if $R$ is essentially of finite type over a field of prime characteristic.

\begin{theorem}[{\cite{HHComparison, HKVfinite, HKV, HKAbelian, HunekeKatz}}]
    \label{theorem: uniform multipliers exist equicharacteristic}
    Let $R$ be a Noetherian normal domain containing a field. Assume either
    \begin{itemize}
        \item $R$ is essentially of finite type over a field;
        \item $R$ is of prime characteristic $p>0$ and $F$-finite.
    \end{itemize}
    If $c\in R$ and $R_c$ is non-singular then there exists constants $C,t$ so that for all ideals $I\subseteq R$, for all $n\in\NN$, if $W$ is the complement of the union of the associated primes of $I$, then $c^{tn}(\overline{I^{Cn}}R_W\cap R) \subseteq \overline{I^n}$.
\end{theorem}

\begin{proof}
By the Uniform Brian\c{c}on-Skoda theorem, it will be enough to show that there exists an element $c$ and constant $C$ so that for all ideals $I$, $c^nI^{(Cn)}\subseteq I^n$, see \cite[Proposition~1.5.2]{SwansonHuneke}. If $R$ is essentially of finite type over a field of characteristic $0$, by \cite[Theorem~4.4 (c)]{HHComparison}, any element belonging to the square of the Jacobian ideal has the desired property. 
    
    Suppose that $R$ is of prime characteristic $p>0$ and $F$-finite and let $F^e_*R$ denote the finitely generated $R$-module obtained via restriction of scalars under the $e$th iterate of the Frobenius map. Then $F_*R$ is generically free and hence there exists a parameter element $c$ so that $F_*R_c$ is a free $R$-module. Replacing $c$ by a suitable power, there exists a free submodule $F_1\subseteq F_*R$ so that $cF_*R\subseteq F_1$. It follows that for each $e\in\mathbb{N}$ that there exists a free module $F_e\subseteq F^e_*R$ such that $c^2F^e_*R \subseteq F_e$, c.f. \cite[Proof of Proposition~3.4]{HKV}. We can replace $c$ by $c^2$ and repeat the argument of \cite[Proof of Theorem~3.5, Page~335, starting at second paragraph]{HKV}, replacing the $\fm$-primary ideal $J$ in \cite[Theorem~3.5]{HKV} with the element $c$, c.f. \cite[Lemma~3.3]{HunekeKatz}.

    The only consideration left is if $k$ is a field of prime characteristic $p>0$, but not necessarily $F$-finite. Let $\Lambda$ denote a $p$-basis of $k^{1/p}$ as a $k$-vector space. There exists a cofinite subset $\Gamma\subseteq \Lambda$ so that if $k^\Gamma=k[\Gamma]$ then $R^\Gamma:=R\otimes_kk^\Gamma$ is an $F$-finite normal domain essentially of finite type over $k^\Gamma$ and $R^\Gamma_c$ is regular, \cite[Theorem~A]{MurayamaGamma}. Therefore $R^\Gamma$ enjoys the claim of the theorem. If $I\subseteq R$ is an ideal of $R$, let $I^\Gamma=IR^\Gamma$. By Lemma~\ref{lemma purely insep field extension}, if $I\subseteq R$ is an ideal, then there is a bijection of Rees valuations $\psi:\RR_I\to \RR_{I^\Gamma}$. Moreover, for all $\nu\in\mathcal{R}_I$, if $\nu^\Gamma=\psi(\nu)$, $\fp_\nu$ the center of $\nu$ in $R$, then $\sqrt{\fp_\nu R^\Gamma}=\fp_{\nu^\Gamma}$ is the center of $\nu^\Gamma$ in $R^\Gamma$, and for all $f\in R$,
    \begin{equation}\label{equation to get ustp in finite field case}
        \frac{\nu(f)}{\nu(I)}=\frac{\nu^\Gamma(f)}{\nu^\Gamma(I^\Gamma)}.
    \end{equation}
    Let $W'$ denote the complement of the union of prime ideals belonging to $\Ass(R/I)\cap \left(\bigcup_{n\in\NN}\Ass(R/\overline{I^n})\right)$, then for every $n\in\NN$,
    \[
    \overline{I^n}R_W\cap R= \overline{I^n}R_{W'}\cap R.
    \]
    By the $F$-finite case of the theorem, we can replace $c$ by a suitable power and there exists constant $C$ so that for all ideals $I\subseteq R$, 
    \[
    c^{n}(\overline{(I^\Gamma)^{Cn}}R^{\Gamma}_{W'}\cap R^\Gamma) \subseteq \overline{(I^\Gamma)^n}.
    \]
    
    The bijection of Rees valuations of $I$ and $I^\Gamma$, equality of values described by (\ref{equation to get ustp in finite field case}), and the valuation criterion of Rees, see Theorem~\ref{theorem standard facts about integral closure and valuations} part (\ref{Valuation Criteria}), imply the same containment properties in $R$, that is for all ideals $I\subseteq R$,
     \[
    c^{n}(\overline{I^{Cn}}R_{W}\cap R)=c^{n}(\overline{I^{Cn}}R_{W'}\cap R) \subseteq \overline{I^n}.
    \]
\end{proof}

\begin{corollary}[Main Theorem~\ref{Main Theorem Improved Uniform Chevalley general}]
    \label{Corollary Improved Uniform Chevalley for eft over field}
    Let $k$ be an field and $R$ a normal domain essentially of finite type over $k$. There exists a constant $C$ so that for all ideals $I\subseteq R$ and for all primes $\fq\in\Spec(R)$, if $IR_\fq\cap R\subseteq \fq^{(t)}$ then
        \[
        I^{(Cn)}R_\fq\cap R\subseteq \fq^{(tn)}.
        \]
\end{corollary}

\begin{proof}
    It suffices to verify the ring $R$ enjoys all hypotheses of Theorem~\ref{theorem how to acheive uniform chevalley}:
    \begin{enumerate}
        \item The ring $R$ enjoys the Uniform Izumi-Rees Property by Theorem~\ref{theorem uniform izumi rees eft over a field}.
        \item The ring $R$ enjoys the Uniform Brian\c{c}on-Skoda Property by \cite[Theorem~4.13]{HunekeUniformBounds}. 
        \item By Theorem~\ref{theorem: uniform multipliers exist equicharacteristic}, there exists an element $0\not=c\in R$ and a constant $C$ so that for all ideals $J\subseteq R,$ for all $n\in \NN$, if $W$ is complement of the union of the associated primes of $J$, then $c^n(\overline{J^{Cn}}R_W\cap R)\subseteq \overline{J^n}$.
        \item The containment $(c)\subseteq R$ enjoys the Uniform Artin-Rees Property by \cite[Theorem~4.12]{HunekeUniformBounds}.
    \end{enumerate}  
    Therefore $R$ enjoys the hypotheses of Theorem~\ref{theorem how to acheive uniform chevalley}, the conclusion of which implies the statement of Main Theorem~\ref{Main Theorem Improved Uniform Chevalley general}.
\end{proof}

\bibliographystyle{skalpha}
\bibliography{main}

\providecommand{\bysame}{\leavevmode\hbox to3em{\hrulefill}\thinspace}
\providecommand{\MR}{\relax\ifhmode\unskip\space\fi MR}
\providecommand{\MRhref}[2]{%
  \href{http://www.ams.org/mathscinet-getitem?mr=#1}{#2}
}
\providecommand{\href}[2]{#2}
\begin{thebibliography}{DSGJ22}

\bibitem[Ben70]{Bennett}
{\sc B.~M. Bennett}: \emph{On the characteristic functions of a local ring},
  Ann. of Math. (2) \textbf{91} (1970), 25--87. {\sf\scriptsize 252388}

\bibitem[Bha21]{BhattCM}
{\sc B.~Bhatt}: \emph{Cohen-macaulayness of absolute integral closures}, 2021.

\bibitem[CS22]{CutkoskySarkar}
{\sc S.~D. Cutkosky and P.~Sarkar}: \emph{Multiplicities and mixed
  multiplicities of arbitrary filtrations}, Res. Math. Sci. \textbf{9} (2022),
  no.~1, Paper No. 14, 33. {\sf\scriptsize 4389501}

\bibitem[DSGJ22]{DSGJ}
{\sc A.~De~Stefani, E.~Grifo, and J.~Jeffries}: \emph{A uniform {C}hevalley
  theorem for direct summands of polynomial rings in mixed characteristic},
  Math. Z. \textbf{301} (2022), no.~4, 4141--4151. {\sf\scriptsize 4449742}

\bibitem[Die10]{Dietz}
{\sc G.~D. Dietz}: \emph{A characterization of closure operations that induce
  big {C}ohen-{M}acaulay modules}, Proc. Amer. Math. Soc. \textbf{138} (2010),
  no.~11, 3849--3862. {\sf\scriptsize 2679608}

\bibitem[ELS01]{ELS}
{\sc L.~Ein, R.~Lazarsfeld, and K.~E. Smith}: \emph{Uniform bounds and symbolic
  powers on smooth varieties}, Invent. Math. \textbf{144} (2001), no.~2,
  241--252. {\sf\scriptsize 1826369}

\bibitem[EH79]{EisenbudHochster}
{\sc D.~Eisenbud and M.~Hochster}: \emph{A {N}ullstellensatz with nilpotents
  and {Z}ariski's main lemma on holomorphic functions}, J. Algebra \textbf{58}
  (1979), no.~1, 157--161. {\sf\scriptsize 535850}

\bibitem[HH92]{HHAnnals}
{\sc M.~Hochster and C.~Huneke}: \emph{Infinite integral extensions and big
  {C}ohen-{M}acaulay algebras}, Ann. of Math. (2) \textbf{135} (1992), no.~1,
  53--89. {\sf\scriptsize 1147957}

\bibitem[HH02]{HHComparison}
{\sc M.~Hochster and C.~Huneke}: \emph{Comparison of symbolic and ordinary
  powers of ideals}, Invent. Math. \textbf{147} (2002), no.~2, 349--369.
  {\sf\scriptsize 1881923}

\bibitem[HS01]{HublSwanson}
{\sc R.~H\"{u}bl and I.~Swanson}: \emph{Discrete valuations centered on local
  domains}, J. Pure Appl. Algebra \textbf{161} (2001), no.~1-2, 145--166.
  {\sf\scriptsize 1834082}

\bibitem[Hun92]{HunekeUniformBounds}
{\sc C.~Huneke}: \emph{Uniform bounds in {N}oetherian rings}, Invent. Math.
  \textbf{107} (1992), no.~1, 203--223. {\sf\scriptsize 1135470}

\bibitem[HK19]{HKAbelian}
{\sc C.~Huneke and D.~Katz}: \emph{Uniform symbolic topologies in abelian
  extensions}, Trans. Amer. Math. Soc. \textbf{372} (2019), no.~3, 1735--1750.
  {\sf\scriptsize 3976575}

\bibitem[HK24]{HunekeKatz}
{\sc C.~Huneke and D.~Katz}: \emph{Uniform symbolic topologies and
  hypersurfaces}, Acta Math. Vietnam. \textbf{49} (2024), no.~1, 99--113.
  {\sf\scriptsize 4754259}

\bibitem[HKV09]{HKV}
{\sc C.~Huneke, D.~Katz, and J.~Validashti}: \emph{Uniform equivalence of
  symbolic and adic topologies}, Illinois J. Math. \textbf{53} (2009), no.~1,
  325--338. {\sf\scriptsize 2584949}

\bibitem[HKV15]{HKVfinite}
{\sc C.~Huneke, D.~Katz, and J.~Validashti}: \emph{Uniform symbolic topologies
  and finite extensions}, J. Pure Appl. Algebra \textbf{219} (2015), no.~3,
  543--550. {\sf\scriptsize 3279373}

\bibitem[Izu85]{Izumi}
{\sc S.~Izumi}: \emph{A measure of integrity for local analytic algebras},
  Publ. Res. Inst. Math. Sci. \textbf{21} (1985), no.~4, 719--735.
  {\sf\scriptsize 817161}

\bibitem[Lip82]{LipmanEquimultiplicity}
{\sc J.~Lipman}: \emph{Equimultiplicity, reduction, and blowing up},
  Commutative algebra ({F}airfax, {V}a., 1979), Lect. Notes Pure Appl. Math.,
  vol.~68, Dekker, New York, 1982, pp.~111--147. {\sf\scriptsize 655801}

\bibitem[MS18]{MaSchwedeSymbolic}
{\sc L.~Ma and K.~Schwede}: \emph{Perfectoid multiplier/test ideals in regular
  rings and bounds on symbolic powers}, Invent. Math. \textbf{214} (2018),
  no.~2, 913--955. {\sf\scriptsize 3867632}

\bibitem[Mat89]{Matsumura}
{\sc H.~Matsumura}: \emph{Commutative ring theory}, second ed., Cambridge
  Studies in Advanced Mathematics, vol.~8, Cambridge University Press,
  Cambridge, 1989, Translated from the Japanese by M. Reid. {\sf\scriptsize
  1011461}

\bibitem[McA80]{McAdamAsymptotic}
{\sc S.~McAdam}: \emph{Asymptotic prime divisors and analytic spreads}, Proc.
  Amer. Math. Soc. \textbf{80} (1980), no.~4, 555--559. {\sf\scriptsize 587926}

\bibitem[Mur21]{MurayamaGamma}
{\sc T.~Murayama}: \emph{The gamma construction and asymptotic invariants of
  line bundles over arbitrary fields}, Nagoya Math. J. \textbf{242} (2021),
  165--207. {\sf\scriptsize 4250735}

\bibitem[Mur23]{MuryamaSymbolic}
{\sc T.~Murayama}: \emph{Uniform bounds on symbolic powers in regular rings},
  2023.

\bibitem[Nag75]{NagataLocalRings}
{\sc M.~Nagata}: \emph{Local rings}, Robert E. Krieger Publishing Co.,
  Huntington, NY, 1975, Corrected reprint. {\sf\scriptsize 460307}

\bibitem[RG18]{RG}
{\sc R.~R.~G.}: \emph{Closure operations that induce big {C}ohen-{M}acaulay
  algebras}, J. Pure Appl. Algebra \textbf{222} (2018), no.~7, 1878--1897.
  {\sf\scriptsize 3763288}

\bibitem[Rat84]{RatliffAsymptoticPrimesIntegral}
{\sc L.~J. Ratliff, Jr.}: \emph{On asymptotic prime divisors}, Pacific J. Math.
  \textbf{111} (1984), no.~2, 395--413. {\sf\scriptsize 734863}

\bibitem[Ree56]{ReesValuationsAssociatedToIdeals}
{\sc D.~Rees}: \emph{Valuations associated with ideals}, Proc. London Math.
  Soc. (3) \textbf{6} (1956), 161--174. {\sf\scriptsize 77513}

\bibitem[Ree61]{ReesDegree}
{\sc D.~Rees}: \emph{Degree functions in local rings}, Proc. Cambridge Philos.
  Soc. \textbf{57} (1961), 1--7. {\sf\scriptsize 124353}

\bibitem[Ree81]{Rees1981}
{\sc D.~Rees}: \emph{Rings associated with ideals and analytic spread}, Math.
  Proc. Cambridge Philos. Soc. \textbf{89} (1981), no.~3, 423--432.
  {\sf\scriptsize 602297}

\bibitem[Ree89]{ReesIzumisTheorem}
{\sc D.~Rees}: \emph{Izumi's theorem}, Commutative algebra ({B}erkeley, {CA},
  1987), Math. Sci. Res. Inst. Publ., vol.~15, Springer, New York, 1989,
  pp.~407--416. {\sf\scriptsize 1015531}

\bibitem[Sal89]{Sally}
{\sc J.~D. Sally}: \emph{One-fibered ideals}, Commutative algebra ({B}erkeley,
  {CA}, 1987), Math. Sci. Res. Inst. Publ., vol.~15, Springer, New York, 1989,
  pp.~437--442. {\sf\scriptsize 1015533}

\bibitem[Sey72]{Seydi}
{\sc H.~Seydi}: \emph{La th\'{e}orie des anneaux japonais}, Colloque
  d'{A}lg\`ebre {C}ommutative ({R}ennes, 1972), Univ. Rennes, Rennes, 1972,
  pp.~Exp. No. 12, 82. {\sf\scriptsize 366896}

\bibitem[SH06]{SwansonHuneke}
{\sc I.~Swanson and C.~Huneke}: \emph{Integral closure of ideals, rings, and
  modules}, London Mathematical Society Lecture Note Series, vol. 336,
  Cambridge University Press, Cambridge, 2006. {\sf\scriptsize 2266432}

\bibitem[TY08]{TakagiYoshida}
{\sc S.~Takagi and K.-i. Yoshida}: \emph{Generalized test ideals and symbolic
  powers}, Michigan Math. J. \textbf{57} (2008), 711--724, Special volume in
  honor of Melvin Hochster. {\sf\scriptsize 2492477}

\end{thebibliography}
\end{document}